\documentclass[11pt]{article}
\usepackage{color}
\usepackage{amsmath,amsthm,amssymb,amsfonts}
\usepackage{epsfig}
\usepackage{graphicx}
\usepackage{bm}
\usepackage{dsfont}
\def\R{\mathbb R}
\def\N{\mathbb N}

\def\e{\varepsilon}
\def\trait (#1) (#2) (#3){\vrule width #1pt height #2pt depth #3pt}
\def\fin{\hfill\trait (0.1) (5) (0) \trait (5) (0.1) (0) \kern-5pt
\trait (5) (5) (-4.9) \trait (0.1) (5) (0)}

\numberwithin{equation}{section}

\newcommand{\be}{\begin{equation}}
\newcommand{\ee}{\end{equation}}
\newcommand{\baa}{\begin{array}}
\newcommand{\eaa}{\end{array}}
\newcommand{\ba}{\begin{eqnarray}}
\newcommand{\ea}{\end{eqnarray}}

\newcommand{\ban}{\begin{eqnarray*}}
\newcommand{\ean}{\end{eqnarray*}}

\newtheorem{theo}{\bf Theorem}[section]
\newtheorem{lem}[theo]{\bf Lemma}
\newtheorem{pro}[theo]{\bf Proposition}
\newtheorem{cor}[theo]{\bf Corollary}
\newtheorem{defi}[theo]{\bf Definition}
\newtheorem{rem}[theo]{\bf Remark}

\def\ds{\rightarrow}
\def\e{\varepsilon}

\def\Ac{{\cal A}}
\def\Tc{{\cal T}}

\def\VF{\mathbf{U}}
\def\VFp{\mathbf{U}^+}
\def\VFm{\mathbf{U}^-}

\def\sb{{\bar s}}

\def\apg{\left\{}
\def\chg{\right\}}

\def\ch{\right.}

\def\a{a}
\def\aluno{\alpha_1}
\def\aldue{\alpha_2}

\def\A{{\cal A}}

\def\Ta{{\cal T}_{x_0}}
\def\Treg{{\cal T}^{\rm reg}_{x_0}}
\def\HTreg{{ H}^{\rm reg}_T}
\def\Aoreg{A_0^{\rm reg}(x_0)}
\def\Aoregx{A_0^{\rm reg}(x)}

\def\eps{\varepsilon}
\renewcommand{\H}{\mathcal{H}}

\def\Euno{\mathcal{E}_1}
\def\Edue{\mathcal{E}_2}
\def\Eh{\mathcal{E}_\H}

\newcommand{\mB}{\mathcal{B}}
\newcommand{\cob}{\overline{\mathop{\rm co}}}
\def\ue{u_\eps}
\def\xe{x_\eps}
\def\vpe{\varphi_\eps}

\newcommand{\trajxo}{X_{x_0}}
\newcommand{\trajxoN}{X_{x_0,N}}
\def\dottrajxo{\dot{X}_{x_0}}

\def\trajyali{Y_{x_0}^{i}}
\def\tiltrajyaluno{\tilde Y_{x_0}^{1}}
\def\trajyaluno{Y_{x_0}^{1}}

\def\dottrajyali{{\dot Y}_{x_0}^{i}}
\def\trajyxki{Y_{x_k}^{i}}
\def\trajyxei{Y_{x_\eps}^{i}}

\def\trajyxbari{Y_{\bar{x}}^{i}}
\def\HT{H_T}

\begin{document}

\title{\bf A Bellman approach for two-domains optimal control problems in $\R^N$}
\author{G.Barles, A. Briani, E. Chasseigne \thanks{Laboratoire de Math\'ematiques et Physique Th\'eorique (UMR CNRS 7350), F\'ed\'eration Denis Poisson (FR CNRS 2964),
  Universit\'e Fran\c{c}ois Rabelais, Parc de Grandmont,
  37200 Tours, France. Email: Guy.Barles@lmpt.univ-tours.fr, Ariela.Briani@lmpt.univ-tours.fr, Emmanuel.Chasseigne@lmpt.univ-tours.fr . }
  \thanks{This work was partially supported by the EU under the 7th Framework Programme Marie Curie
Initial Training Network ``FP7-PEOPLE-2010-ITN'', SADCO project, GA number 264735-SADCO.}
  }
\maketitle

\begin{abstract}
This article is the starting point of a series of works whose aim is the study of deterministic control problems where the dynamic and the running cost can be completely different in two (or more) complementary domains of the space $\R^N$. As a consequence, the dynamic and running cost present discontinuities at the boundary of these domains and this is the main difficulty of this type of problems. We address these questions by using a Bellman approach: our aim is to investigate how to define properly the value function(s), to deduce what is (are) the right Bellman Equation(s) associated to this problem (in particular what are the conditions on the set where the dynamic and running cost are discontinuous) and  to study the uniqueness properties for this Bellman equation. In this work, we provide rather complete answers to these questions in the case of a simple geometry, namely when we only consider two different domains which are half spaces: we properly define the control problem, identify the different conditions on the hyperplane where the dynamic and the running cost are discontinuous and discuss the uniqueness properties of the Bellman problem by either providing explicitly the minimal and maximal solution or by giving explicit conditions to have uniqueness.
\end{abstract}

 \noindent {\bf Key-words}: Optimal control, discontinuous dynamic, Bellman Equation, viscosity solutions.
\\
{\bf AMS Class. No}:
49L20,   
49L25,   
35F21. 

\section{Introduction}

In this paper, we consider infinite horizon control problems where we have different dynamics and running costs in the half-spaces $\Omega_1:=\{  x \in \R^N  \: : \: x_N >0 \}$ and $\Omega_2:=\{  x \in \R^N  \: : \: x_N  < 0 \}$.

On each domain $\Omega_i$ ($i=1,2$), we have a controlled dynamic given by $b_i : \overline{\Omega}_i \times A_i \to \R^N$, where $A_i$ is the compact metric space where the control takes its values and a running cost $l_i : \overline{\Omega}_i \times A_i \to \R$.  We assume that these dynamics and running costs satisfy standard assumptions: the functions $b_i(\cdot,\alpha_i), l_i(\cdot,\alpha_i)$ are continuous and uniformly bounded and the $b_i(\cdot,\alpha_i)$ are equi-Lipschitz continuous. To simplify the exposure, we also suppose that the system is controllable on both sides.

The first difficulty is to define the controlled dynamic and in particular for trajectories which may stay for a while on the hyperplane $\H:= \overline{\Omega}_1 \cap \overline{\Omega}_2 = \apg  x \in \R^N  \: : \: x_N=0 \chg$. To do so, we follow the pioneering work of Filippov~\cite{Fi} and use the approach through differential inclusions. As a consequence, we see that in particular there exist trajectories which stay on $\H$ at least for a while. Such trajectories are build through a dynamic of the form
$$b_\H(x,(\alpha_1, \alpha_2, \mu)):=\mu b_1 (x,\alpha_1) + (1-\mu)b_2(x,\alpha_2)\,,$$
for $x\in \H$, with $\mu \in [0,1]$, $\alpha_i \in A_i$ and $b_\H(x,(\alpha_1, \alpha_2, \mu))\cdot e_N = 0$ where $e_N:=(0,\cdots,0,1)$. We denote by $A_0(x)$ the set of such controls $\a:=(\alpha_1, \alpha_2, \mu)$. The associated cost is
$$l_\H(x,\a)=l_\H(x,(\alpha_1, \alpha_2, \mu)):=\mu l_1 (x,\alpha_1) + (1-\mu)l_2(x,\alpha_2)\,.$$
Once this is done, we can define value-functions and look for the natural Bellman problem(s) which are satisfied by these value functions. Actually we are going to define two value functions, we come back on this point later on.

It is well-known that, for classical infinite horizon problems, i.e. here in $\Omega_1$ and $\Omega_2$, the equations can be written as
\begin{equation}\label{Bellman-Om}
\begin{array}{cc}
H_1(x,u,Du) = 0     &   \hbox{ in   }\Omega_1\,, \\
H_2(x,u,Du) = 0     &     \hbox{ in   }\Omega_2\,,
\end{array}
\end{equation}
where $H_1, H_2$ are the classical  Hamiltonians
\be  \label{def:Ham}
H_i(x,u,p):=\sup_{ \alpha_i \in A_i} \apg  -b_i(x,\alpha_i) \cdot p+ \lambda u - l_i(x,\alpha_i) \chg\,,
\ee
where $\lambda >0$ is the actualization factor.
From viscosity solutions' theory, 
 it is natural to think that we have to complement these equations by
\begin{equation}\label{Bellman-H-sub}
\min\{ H_1(x,u,Du),  H_2(x,u,Du)\}\leq 0  \quad\hbox{on   } \H \; ,
\end{equation}
\begin{equation}\label{Bellman-H-sup}
 \max\{ H_1(x,u,Du),  H_2(x,u,Du)\}\geq 0   \quad\hbox{on   } \H \; .
\end{equation}
This is actually true since the two value functions we introduce naturally satisfy such inequalities. Note that, for the sake of simplicity, we always say a {\it sub and supersolution of   \eqref{Bellman-Om}-\eqref{Bellman-H-sub}-\eqref{Bellman-H-sup}},  while it has to be understood that both verify \eqref{Bellman-Om} in $\Omega_1$ and $\Omega_2$, but a subsolution only satisfies \eqref{Bellman-H-sub} on $\H$ while a supersolution only satisfies
\eqref{Bellman-H-sup} on $\H$.

The main interesting questions are  then
\begin{itemize}
\item[1)] Does  problem (\ref{Bellman-Om})-(\ref{Bellman-H-sub})-(\ref{Bellman-H-sup})  have a unique solution~?
\item[2)]  Do the value functions satisfy other properties on $\H$ ?
\item[3)]  Do these extra properties allow to characterize each of the value function either as the unique solution of a Bellman problem or at least as the minimal supersolution or the maximal subsolution of them?
\end{itemize}
Our results give complete answers to the above questions.

Concerning Question 1), the answer is no in general. We do not have uniqueness for the problem (\ref{Bellman-Om})-(\ref{Bellman-H-sub})-(\ref{Bellman-H-sup}) but we can identify the maximal subsolution (and solution) and the minimal supersolution (and solution) of (\ref{Bellman-Om})-(\ref{Bellman-H-sub})-(\ref{Bellman-H-sup}): they are value functions of suitable control problems which we are going to define now. The difference between the two is related to the possibility of accepting or rejecting some strategies on $\H$. To be more precise, we call \textit{singular} a dynamic $b_\H(x,\a)$ on $\H$ (i.e. such that $b_\H(x,\a) \cdot e_N=0$ ) when $b_1 (x,\alpha_1)\cdot e_N > 0$ and $b_2(x,\alpha_2)\cdot e_N < 0$ while the \textit{non-singular} (or \textit{regular}) ones are those for which the $b_i (x,\alpha_i)\cdot e_N$ have the opposite (may be non strict) signs. Then, the minimal solution $\VFm$ is obtained when allowing all kind of controlled strategies (with singular and regular dynamics) while the maximal solution $\VFp$ is obtained by forbidding singular dynamics. The uniqueness problem comes from the fact that, in some sense, the singular strategies are not encoded in the equations (\ref{Bellman-Om})-(\ref{Bellman-H-sub})-(\ref{Bellman-H-sup}), while it is the case for the regular ones.

For Question 2), the answer is the following: if we allow any kind of controlled strategies, both the regular and the singular ones, the associated value function, namely $\VFm$, also satisfies the inequality
\begin{equation}\label{Bellman-H}
\HT (x,u,D_\H u) \leq 0   \quad\hbox{on   }\H\; ,
\end{equation}
where $D_\H u:=(\frac{\partial u}{\partial x_1}, \cdots, \frac{\partial u}{\partial x_{n-1}})$ is the gradient of $u$ with respect to the $\H$-variables $x_1,\cdots,x_{n-1}$ and, for $x\in \H, u\in \R, p' \in \R^{N-1}$, $ \HT(x,u,p')$ is given by
$$
\sup_{A_0(x)}   \{   -b_\H (x,\a)\cdot (p',0)
+  \lambda u - l_\H (x,\a)  \}\,.
$$
We emphasize the fact that this viscosity inequality is actually a $\R^{N-1}$ viscosity inequality (meaning that we are considering maximum points relatively to $\H$ and not to $\R^N$); it reflects the suboptimality of the controlled trajectories which stay on $\H$. This inequality makes a difference between $\VFm$ and $\VFp$ since $\VFp$ satisfies the same inequality but with $A_0(x)$ being replaced by $A_0^{\rm reg}(x)$ consisting in elements of $A_0(x)$ satisfying $b_1 (x,\alpha_1)\cdot e_N \leq 0$ and $b_2 (x,\alpha_2)\cdot e_N \geq 0$.

For Question 3), (\ref{Bellman-H}) also makes a difference since there exists a unique solution of (\ref{Bellman-Om})-(\ref{Bellman-H-sub})-(\ref{Bellman-H-sup})-(\ref{Bellman-H}). In other words, the uniqueness gap for (\ref{Bellman-Om})-(\ref{Bellman-H-sub})-(\ref{Bellman-H-sup}) just comes from the fact that a subsolution of (\ref{Bellman-Om})-(\ref{Bellman-H-sub})-(\ref{Bellman-H-sup}) does not necessarily satisfy (\ref{Bellman-H}) and this is due to the difficulty to take into account (at the equation level) some singular strategies. We illustrate this fact by an explicit example in dimension $1$.

Besides of the answers to these three questions, we provide the complete structure of solutions in $1$-D and we also study the convergence of natural approximations.

We end by remarking that there are rather few articles on the same topic, at least if we insist on having such a structure with a general discontinuous dynamic. A pioneering work is the one of Dupuis \cite{Du} that considers a similar method to construct a numerical method for a calculus of variation problem with discontinuous integrand. The work of Bressan and Hong \cite{BrYu} goes in the same direction by studying an optimal control problem on stratified domains. Problems with a discontinuous running cost were addressed by either Garavello and Soravia \cite{GS1,GS2}, or Camilli and Siconolfi   \cite{CaSo} (even in an $L^\infty$-framework) and Soravia \cite{So}. To the best of our knowledge, all the uniqueness results use a special structure of the discontinuities as in \cite{DeZS,DE,GGR} or an hyperbolic approach as in \cite{AMV,CR}. We finally remark that problems on network (see \cite{IMZ},\cite{ACCT}, \cite{ScCa}) share the same kind of difficulties.

The paper is organized as follows: in Section~\ref{sec:control}, we show how to define the dynamic and cost of the control problem in a proper way, we introduce two different value functions ($\VFm$ and  $\VFp$) and, in Theorem~\ref{teo:HJ}, we show that they are solutions of (\ref{Bellman-Om})-(\ref{Bellman-H-sub})-(\ref{Bellman-H-sup}). In addition, we prove that $\VFm$ satisfies the subsolution inequality (\ref{Bellman-H}) while $\VFp$ satisfies a less restrictive inequality, associated to the Hamiltonian involving only {\em regular} controls $\HTreg \leq \HT$. Section~\ref{sec:Prop-sub-super} is devoted to study the properties of {\em any} sub and supersolution of (\ref{Bellman-H-sub})-(\ref{Bellman-H-sup})-(\ref{Bellman-H}) and, in particular, the additional inequalities that they satisfy on $\H$ (inequalities which are connected to $ \HT$ or $\HTreg$). In Section~\ref{sec:uni}, we use these properties to provide a comparison result for (\ref{Bellman-Om})-(\ref{Bellman-H-sub})-(\ref{Bellman-H-sup})-(\ref{Bellman-H}) (Theorem~\ref{Uni-dim-N}); one of the main consequences of this result is that $\VFm$ is the minimal supersolution and solution of (\ref{Bellman-Om})-(\ref{Bellman-H-sub})-(\ref{Bellman-H-sup}), while $\VFp$ is the maximal subsolution and solution of (\ref{Bellman-Om})-(\ref{Bellman-H-sub})-(\ref{Bellman-H-sup}) (cf. Corollary~\ref{sous-max-sur-min}). In Section~\ref{sect:oneD}, we study in details the case of the dimension~$1$ by providing the complete structure of the solutions, together with examples of different behaviors.  Finally Section~\ref{sec:approx} is devoted to examine the effect of several approximations (Filippov and vanishing viscosity).

\section{A control  problem}\label{sec:control}

The aim of this section is to give a sense to infinite horizon control problems which have different dynamic and cost in $\Omega_1:=\{  x \in \R^N  \: : \: x_N >0 \}$ and in $\Omega_2:=\{  x \in \R^N  \: : \: x_N  < 0 \}$. Of course, the difficulty is to understand how to define the problem on $\H:= \overline{\Omega}_1 \cap \overline{\Omega}_2 = \apg  x \in \R^N  \: : \: x_N=0 \chg$.

We first describe the assumptions on the dynamic and cost in each $\Omega_i\ (i=1,2)$. On $\Omega_i$, the sets of controls are denoted by $A_i$, the system is driven by a dynamic $b_i$ and the cost is given by $l_i$.

Our main assumptions are the following

\begin{itemize}

\item[{[H0]}]  For $i=1,2$, $A_i$ is a compact metric space and $b_i : \R^N \times A_i  \ds \R^N$ is a continuous bounded function. Moreover there exists $L_i \in \R$ such that, for any $x,y \in \R^N$ and $\alpha_i \in A_i$
$$ |b_i(x,\alpha_i)-b_i(y,\alpha_i)|\leq L_i |x-y|\; .$$

\item[{[H1]}]  For $i=1,2$, the function $l_i : \R^N \times A_i  \ds \R^N$ is a continuous, bounded function.  \\

\item[{[H2]}] For each $x \in \R^N$, the sets $\apg   (b_i(x, \alpha_i),l_i(x, \alpha_i))  \:  :     \:   \alpha_i \in A_i \chg$,  ($i=1,2$), are closed and convex. Moreover there is a $\delta>0$ such that for any $i=1,2$ and $x\in\R^N$, 
\begin{equation}\label{cont-ass}
\overline{B(0,\delta)} \subset B_i(x):= \apg   b_i(x, \alpha_i)  \:  :     \:   \alpha_i \in A_i \chg.
\end{equation}
\end{itemize}

Assumptions [H0], [H1] are the classical hypotheses used in infinite horizon control problems. We have strengthened them in [H2] in order to keep concentrated in the main issues of the problem.  Indeed,  the first part of assumption [H2] avoids the use of relaxed controls, while the second part is a controllability assumption which will lead us to Lipschitz continuous value functions. In a forthcoming work, we are going to weaken [H2] by assuming only some kind of controlability in the normal direction: this weaker assumption is inspired from \cite{BP3} where it is used to obtain comparison results for {\em discontinuous} sub and super-solutions of exit time-Dirichlet problems without assuming the ``cone's condition'' of Soner \cite{Son}. In our framework, 
 $\H$ plays a similar role as the boundary of the domain in exit time problems since one of the main question is how the trajectories of the dynamics reach $\H$ and how the value function behaves on $\H$. In \cite{So}, Soravia uses a transversality condition which looks like the ``cone's condition'' of Soner \cite{Son} to prove comparison results while [H2] or its weaker version are more related to the Barles-Perthame approach \cite{BP3} (See also \cite{Ba}).

In order to define  the optimal control problem in all $\R^N$, we first have to define the dynamic and therefore we are led to consider an ordinary differential equation with discontinuous right-hand side. This kind of ode has been treated for the first time in the pioneering work of Filippov~\cite{Fi}. We are going to define the trajectories of our optimal control problem by using the approach through differential inclusions which is rather convenient here. This  approach has been introduced in \cite{Wa} (see also \cite{AF}) and has become now classical.  To do so in a more general setting, and since the controllability condition (\ref{cont-ass}) plays no role in the definition of the dynamic, we are going to use Assumption [H2]$_{nc}$ which is [H2] without (\ref{cont-ass}).

Our trajectories $\trajxo(\cdot)=\big(X_{x_0,1},X_{x_0,2},\dots,\trajxoN\big)(\cdot)$ are Lipschitz continuous functions which are solutions  of the following differential inclusion
\begin{equation} \label{def:traj}
\dottrajxo (t) \in \mB(\trajxo (t))  \quad \hbox{for a.e.  } t \in (0,+\infty)  \: ; \quad \trajxo (0)=x_0
\end{equation}
where
\[
\mB(x):= \apg
\begin{array}{cc}
  B_1(x)   &   \mbox{  if }  x_N > 0\,,    \\
  B_2(x) &    \mbox{  if }   x _N< 0\,,   \\
   \cob \big( B_1(x) \cup B_2(x) \big)  &        \mbox{  if }  x_N=0\,,
\end{array}
\ch
\]
the notation $\cob(E)$ referring to the convex closure of the set $E\subset\R^N$.
We point out that if the definition of $\mB(x)$ is natural if either $x_N > 0$ or $x_N < 0$, it is dictated by the assumptions to obtain the existence of a solution to (\ref{def:traj}) for $x_N = 0$ (see below).

In the sequel, we use the set $A:= A_1 \times A_2 \times [0,1]$ where the control function really takes values and we set $\Ac := L^\infty (0,+\infty; A)$.
We have the following
\begin{theo}\label{def:dyn}
Assume [H0], [H1] and [H2]$_{nc}$. Then

\noindent $(i)$ For each $x_0 \in \R^N$, there exists a Lipschitz function $\trajxo : [0,\infty[ \ds \R^N$
which is a solution of  the differential inclusion  \eqref{def:traj}.

\noindent $(ii)$ For each solution  $\trajxo(\cdot)$ of   \eqref{def:traj},  there exists a control
$\a(\cdot)=\big(\alpha_1(\cdot),\alpha_2(\cdot),\mu(\cdot)\big) \in \Ac$ such that
\begin{equation}\begin{aligned}\label{fond:traj}
\dottrajxo (t) &= b_1\big(\trajxo (t),\alpha_1(t)\big)\mathds{1}_{\apg \trajxo (t) \in \Omega_1 \chg }+
b_2\big(\trajxo (t),\alpha_2(t)\big)\mathds{1}_{\apg  \trajxo (t) \in \Omega_2 \chg } \\[2mm]
&+ b_\H\big(\trajxo (t) , \a(t)\big)\mathds{1}_{\apg  \trajxo (t) \in \H \chg }\,  \quad   \hbox{ for a.e. }   \: t  \in \R^+ \,  ,
\end{aligned}\end{equation}
(where $\mathds{1}_{A}(\cdot)$ stands for the indicator function of the set $A$.)

\noindent $(iii)$ If $e_N = (0,\cdots,0,1)$, then
$$b_\H\big(\trajxo (t) ,\a(t)\big)\cdot e_N = 0 \quad \hbox{a.e. on  }\{\trajxoN(t) = 0\}\; .$$
\end{theo}

\begin{proof} This result follows from two classical results  in \cite{AF}. \\

\noindent\textsc{Step 1 --}
Since the set-valued map $\mB$ is upper semi-continuous with convex compact images, thanks to \cite[Theorem 10.1.3]{AF}, we have that, for each $x_0 \in \R^N$, there exists an absolutely continuous solution $\trajxo(\cdot)$, of the differential inclusion
 \eqref{def:traj}, i.e.
 \[
\dottrajxo(t) \in \mB(\trajxo (t))  \mbox{  for a.e.  }  t \in (0,\infty)\,;    \quad \trajxo(0)=x_0\,.
 \]
Note that the solution is defined in  all $\R^ +$ and Lipschitz continuous, thanks to the boundedness of $\mB$.
This first step justifies the definition of $\mB$ for $x_N=0$.\\

\noindent\textsc{Step 2 --} The next step consists in applying Filippov's Lemma (cf. \cite[Theorem 8.2.10]{AF}).
To do so, we define the map $g:\R^+ \times A \rightarrow \R^N$ as follows
$$
g(t,\a):=
\apg
\begin{array}{cc}
    b_1\big(\trajxo (t) ,\alpha_1\big)  &   \mbox{  if }  \trajxoN(t) >0    \\
    b_2\big(\trajxo (t) ,\alpha_2\big) &    \mbox{  if }  \trajxoN(t)< 0   \\
    b_\H\big(\trajxo (t) ,a \big)   &        \mbox{  if }  \trajxoN(t)=0\,,
 \end{array}
\ch
$$
where $a=(\alpha_1,\alpha_2,\mu)$. We point out that we use here the general definition of $ b_\H$, {\em without assuming that $b_\H\big(\trajxo (t) ,a \big)\cdot e_N=0$}.

%
We claim that $g$ is a Caratheodory map. Indeed, it is first clear that, for fixed $t$, the function $a\mapsto g(t,a)$
is continuous. Then, to check that $g$ is measurable with respect to its first argument we fix $a\in A$,
an open set $\mathcal{O}\subset\R^N$ and evaluate
$$ g^{-1}_a(\mathcal{O})=\big\{ t>0: g(t,a)\cap\mathcal{O}\neq\emptyset\big\}
$$
that we split into three components, the first one being
$$ g^{-1}_a(\mathcal{O})\cap \{t:\ \trajxoN(t) < 0\}= \big\{ t>0: b_1(\trajxo (t) ,\alpha_1)\in\mathcal{O}\big\}\cap \{t:\ \trajxoN(t) < 0\}\,. $$
Since the function $t\mapsto b_1(\trajxo (t),\alpha_1)$ is continuous, this set is the intersection of open sets, hence it is open and therefore measurable. The same argument works for the other components, namely $\{t:\ \trajxoN(t) <0\}$ and $\{t:\ \trajxoN(t)= 0\}$ which finishes the claim.

The function $t\mapsto\dottrajxo(t)$ is measurable and, for any $t$, the differential inclusion implies that
$$\dottrajxo (t) \in g(t,A)\; ,$$
therefore, by Filippov's Lemma, there exists a measurable map $a(\cdot)=(\aluno,\aldue,\mu)(\cdot) \in \Ac$ such that \eqref{fond:traj}  is fulfilled.
In particular, by the definition of $g$, we have for a.e. $t \in \R_*^+ $
\be  \label{eq:g}
\dottrajxo (t)= \apg
\begin{array}{cc}
    b_1\big(\trajxo (t) ,\alpha_1(t)\big)  &   \mbox{  if }  \trajxoN(t) >0    \\
    b_2\big(\trajxo (t) ,\alpha_2(t)\big) &    \mbox{  if }  \trajxoN(t)< 0   \\
    b_\H\big(\trajxo (t) ,\a(t)\big)   &  \mbox{  if }  \trajxoN(t)=0.
 \end{array}
\ch
\ee

\noindent\textsc{Step 3 --} The proof of $(iii)$ is an immediate consequence of Stampacchia's Theorem (cf. for example D. Gilbarg and N.S Trudinger \cite{gt}) since, if $y(t):= (\trajxo (t))_N$, then $\dot{y} (t) = 0$ a.e. on the set $\{y(t)=0\}$.
\end{proof}

It is worth remarking that, in Theorem~\ref{def:dyn}, a solution $\trajxo(\cdot) $ can be associated to several controls $a(\cdot)$; indeed in \eqref{fond:traj} or \eqref{eq:g} the associated control is not necessarily unique. To set properly the control problem, without showing that \eqref{eq:g} has a solution for any $a(\cdot)$, we introduce the set $\Tc_{x_0}$ of admissible controlled trajectories starting from the initial datum $x_0$
\begin{equation*}
\Tc_{x_0}:= \big\{  (\trajxo(\cdot),\a(\cdot))\in {\rm Lip}(\R^+;\R^N) \times \A \mbox{ such that  }
\mbox{ \eqref{fond:traj}  is fulfilled and }  \trajxo(0)=x_0   \big\}
\end{equation*}
and we set
$$
  \Euno:= \apg  t  \:   : \: \trajxo (t) \in \Omega_1   \chg,\quad \Edue:= \apg  t  \:   : \: \trajxo (t) \in \Omega_2  \chg,\quad
\Eh:= \apg  t  \:   : \: \trajxo (t) \in \H  \chg\,.
$$
We finally define the set of regular controlled trajectories
$$
\Treg:= \big\{  (\trajxo(\cdot),\a (\cdot)) \in \Ta \mbox{ such that, for almost all }  t\in\Eh,  \: b_\H(\trajxo(t),a(t))   \mbox{ is regular}  \big\}.
$$
Recall that, we call \textit{singular} a dynamic $b_\H(x,\a)$ on $\H$ with $a=(\alpha_1, \alpha_2, \mu)$ when $b_1 (x,\alpha_1)\cdot e_N > 0$ and $b_2(x,\alpha_2)\cdot e_N < 0$, while the \textit{non-singular} (or \textit{regular}) ones are those for which the $b_i (x,\alpha_i)\cdot e_N$ have the opposite (may be non strict) signs.

\noindent {\bf The cost functional.}
Our aim is  to minimize an infinite horizon  cost functional such that we respectively pay $l_{i}$ if the trajectory is in $\Omega_i$, $i=1,2$ and
$l_\H$ if it is on $\H$. \\

More precisely,  the cost associated to $(\trajxo(\cdot) ,\a)  \in \Ta$ is
$$  
J(x_0; (\trajxo, \a)):=\int_0^{+\infty} \ell\big(\trajxo (t),a \big) e^{-\lambda t} dt
$$
where the Lagrangian is given by
$$
\ell(\trajxo (t),a)  :=  l_1(\trajxo (t),\alpha_1(t)) \mathds{1}_{\Euno}(t)
       +  l_2(\trajxo (t),\alpha_2(t)) \mathds{1}_{\Edue }(t)
       +   l_\H(\trajxo (t),\a(t))\mathds{1}_{\Eh }(t)\,.
$$
%
%

\noindent {\bf The value functions. }For each initial data $x_0$, we define the following two value functions
\be \label{eq:valorem}
 \VFm (x_0):= \inf_{(\trajxo,\a) \in \Ta} J(x_0; (\trajxo, \a))
 \ee
\be \label{eq:valorep}
 \VFp (x_0):= \inf_{(\trajxo,\a) \in \Treg} J(x_0; (\trajxo, \a))
 \ee

The first key result is the {\bf Dynamic Programming Principle.}

\begin{theo}
Assume [H0], [H1] and [H2].
Let  $\VFm,\VFp$ be the  value functions defined in  \eqref{eq:valorem} and \eqref{eq:valorep}, respectively.
For each  initial data $x_0$, and each time $ \tau \geq 0$, we have
\begin{equation} 
   \VFm (x_0)= \inf_{  (\trajxo,\a) \in \Ta} \left\{ \int_0^\tau   \ell\big(\trajxo (t),a \big)  e^{-\lambda t} dt
    +  e^{-\lambda \tau}  \VFm (\trajxo(\tau)) \right\}
\end{equation}
\begin{equation} 
   \VFp (x_0)= \inf_{  (\trajxo,\a) \in \Treg} \left\{ \int_0^\tau   \ell\big(\trajxo (t),a \big)  e^{-\lambda t} dt
    +  e^{-\lambda \tau}  \VFp (\trajxo(\tau)) \right\}
\end{equation}
\end{theo}
\begin{proof}
The proof is standard, so we skip it.
\end{proof}

Because of our assumption [H2] on $b_1,b_2$, it follows that $\overline{B(0,\delta)} \subset  \mathcal{B}(x)$ for any $x\in \R^N$.
Hence the system is controllable, which means, roughly speaking,
that the set of admissible controls is rich enough  to avoid "forbidden directions" in any point of $\R^N$.

The most important consequence of this is that both value functions  $\VFm$ and   $\VFp$ are Lipschitz continuous.

\begin{theo}
Assume [H0], [H1] and [H2]. Then, the value functions $\VFm$ and $\VFp$  are  bounded, Lipschitz continuous functions
from $ \R^N$ into $\R$.   \\
\end{theo}
\begin{proof} Since the proof is the same for $\VFm$ and $\VFp$, we denote by $\VF$ a function which can be either $\VFm$ or $\VFp$.
 We first notice that, if $M$ is a large enough constant such that $||l_1||_\infty, ||l_2||_\infty \leq M$ (recall that $l_1, l_2$ are bounded), we have
$$|\VF(z)| \leq \frac{M}{\lambda}\quad \hbox{for any  }z \in \R^N\; ,$$
therefore $\VF$ is bounded.

Next let $x, y\in \R^N$ and set $K:= \frac{2M}{\delta}$. We are going to prove that
$$\VF(x) \leq \VF(y) + K|x-y|\; .$$
Of course, if this inequality is true for any $x,y\in \R^N$, it implies the Lipschitz continuity of $\VF$.

To prove it, we assume that $x\neq y$ (otherwise the inequality is obvious) and we set $e:= \frac{y-x}{|y-x|}$. By [H2], since $\overline{B(0,\delta)} \subset \mathcal{B}(z)$ for any $z\in \R^N$, it follows that $\delta e \in \mathcal{B}(z)$ for any $z\in \R^N$ and the trajectory
$$ X_x(t):= x + \delta e \cdot t\; ,$$
is a solution of the differential inclusion with $X_x(0)= x$ and $X_x(\tau) = y$ with $\tau = \frac{|x-y|}{\delta}$.

By the Dynamic Programming Principle
$$
\VF (x) \leq   \int_0^\tau   \ell\big(X_x (t),a \big)  e^{-\lambda t} dt   +  e^{-\lambda \tau}  \VF (y)\; ,
$$
and estimating the cost $\ell\big(X_x (t),a \big)$ by $M$, we obtain
$$ \VF (x)- \VF (y) \leq M \tau +  (1- e^{-\lambda \tau})  ||\VF||_\infty\; .$$
Finally,
$$ \VF (x)- \VF (y) \leq  M \tau + \lambda \tau  ||\VF||_\infty = \frac{2M}{\delta}|x-y|$$
and the proof is complete.

 \end{proof}

\noindent{\bf The Hamilton-Jacobi-Bellman Equation.}
In order to describe what is happening on the hyperplane $\H$, we shall introduce
two "tangential Hamiltonians" defined on $\H$, namely $H_T,\HTreg:\H\times\R\times\R^{N-1}\to\R$.
We introduce some notations to be clear on how they are defined:
the points of $\H$ will be identified indifferently by $x'\in\R^{N-1}$ or by $x=(x',0)\in\R^N$.
Now, for the gradient variable
we use the decomposition $p=(p_\H,p_N)$ and, when dealing with a function $u$, we shall also use the notation $D_\H u$ for
the ${(N-1)}$ first components of the gradient, i.e., $$D_\H u:=(\frac{\partial u}{\partial x_1}, \cdots, \frac{\partial u}{\partial x_{n-1}})\quad\hbox{and}\quad Du=\Big(D_\H u,\frac{\partial u}{\partial x_N}\Big)\,.$$
Note that, for the sake of consistency of notations, we also denote by $D_\H u$  the gradient of a function $u$ which is only defined on $\R^{N-1}$. 
Then, for any $(x,u,p_\H)\in\H\times\R \times\R^{N-1}$ we set
\begin{equation}  \label{def:HamHT}
\HT(x,u,p_\H):=\sup_{A_0(x)}
 \big\{  - b_\H(x,\a)\cdot (p_\H,0)  + \lambda u -  l_\H(x,\a)   \big\}
\end{equation}
where $A_0(x):=\big\{\a=(\aluno,\aldue,\mu) \in A : b_\H(x,\a) \cdot e_N=0\big\}$ and
\begin{equation}  \label{def:HamHTreg}
\HTreg(x,u,p_\H):=\sup_{\Aoregx} \big\{  - b_\H(x,\a)\cdot (p_\H,0)  + \lambda u -  l_\H(x,\a)   \big\}
\end{equation}
where $\Aoregx:=\big\{\a=(\aluno,\aldue,\mu)  \in A_0(x)\:   ; \:
b_1(x,\aluno) \cdot e_N  \leq 0 \mbox{ and }  b_2(x,\aldue) \cdot e_N  \geq 0  \big\}$.

The definition of viscosity sub and super-solutions
for $H_T$ an $\HTreg$ have to be understood on $\H$, as follows:
\begin{defi}
    A bounded usc function $u:\H\to\R$ is a viscosity subsolution of
    $$H_T(x,u,D_\H u)=0\quad \text{on   } \H$$
    if, for any $\phi\in C^1(\R^{N-1})$ and any maximum point $x_0'$ of 
    $x'\mapsto u(x')-\phi(x')$, one has $$H_T\big(x_0,\phi(x_0'),D_\H\phi(x_0')\big)\leq0\; ,$$ with $x_0=(x_0',0)$.
\end{defi}
A similar definition holds for $\HTreg$, for supersolutions and solutions.
Of course, if $u$ is defined in a bigger set containing $\H$ (typically $\R^N$), we have to use
$u|_\H$ in this definition, a notation that we shall sometimes omit when not necessary.

We first prove that both the value functions  $\VFm$ and  $\VFp$ are viscosity solutions of the Hamilton-Jacobi-Bellman problem (\ref{Bellman-Om})-(\ref{Bellman-H-sub})-(\ref{Bellman-H-sup}), while  they
fulfill different  inequalities on the hyperplane $\H$.

\begin{theo}Assume [H0], [H1] and [H2]. \label{teo:HJ}
The value functions $\VFm$ and   $\VFp$ are  both  viscosity solutions of the Hamilton-Jacobi-Bellman problem
\be   \label{eq:HJ}
\apg
\begin{array}{cc}
H_1(x,u,Du) = 0   \quad\hbox{in  } \Omega_1 &    \\
 H_2(x,u,Du)  = 0  \quad\hbox{in  } \Omega_2 &    \\
  \min\big\{ H_1(x,u,Du),  H_2(x,u,Du)\big\}\leq 0 \quad\hbox{on  } \H &\\
   \max\big\{ H_1(x,u,Du),  H_2(x,u,Du)\big\}\geq 0  \quad\hbox{on  } \H. &
\end{array}
\ch
\ee
Moreover, $x'\mapsto \VFm (x',0)$ verifies
\be  \label{eq:HJUm}
  H_T\big(x,u, D_\H u\big)  \leq 0   \quad\hbox{on  } \H \; ,
\ee
while  $x'\mapsto \VFp (x',0)$  satisfies
\be \label{eq:HJUp}
\HTreg\big(x,u, D_\H u\big)  \leq 0  \quad\hbox{on  } \H \; .
  \ee
\end{theo}

\begin{rem} Once it is proved that $\VFp$ is a viscosity solution of \eqref{Bellman-Om}-\eqref{Bellman-H-sub}-\eqref{Bellman-H-sup},
then \eqref{eq:HJUp} follows directly from Theorem  \ref{teo:sotto}, which concerns all subsolutions of
\eqref{Bellman-Om}-\eqref{Bellman-H-sub}-\eqref{Bellman-H-sup}. However we give below a direct proof for $\VFp$.
\end{rem}

\begin{proof}
We start by proving that $\VFm$ and $\VFp$ are both viscosity  supersolutions of  \eqref{eq:HJ}. Let $\VF=\VFp$ or $\VFm$.
    We consider $\phi\in C^1(\R^N)$ and $x_0 \in \R^N$
    such that $\VF-\phi$ has a local minimum at $x_0$, that is, for some $r>0$ we have
    $$ \VF(x_0)-\phi(x_0)\leq \VF(x)-\phi(x) \quad\hbox{if   } |x-x_0|<r \,.$$
    We assume that this min is zero for simplicity, i.e. $ \VF(x_0)=\phi(x_0)$.

    If  $x_0\in\Omega_1$ (or $\Omega_2$), we can always find a time  $\tau$ small enough so that $|\trajxo (t)-x_0|<r$
    and $\trajxo (t)\in\Omega_1$ (or $\Omega_2$), for $0<t<\tau$.  Therefore the proof in this case is classical and we do not detail it.
     (See \cite{BCD}, \cite{Ba}.)

    Now assume that  $x_0\in\H$ and $\tau$ is small enough so that $|\trajxo (t)-x_0|<r$. By  the Dynamic Programming Principle we have
    \begin{equation}\label{ineq:sursoDPPl}
    \VF(x_0) = \inf_{  (\trajxo,\a) }
    \apg  \int_0^\tau   \ell\big(\trajxo (t),\a (t)\big)  e^{-\lambda t} dt  +  e^{-\lambda \tau}  \VF (\trajxo(\tau))\chg\,,
    \end{equation}
     where the inf is taken over $\Ta$ or $\Treg$ according to whether $\VF=\VFm$ or $\VFp$. Thus
    \begin{equation}\label{ineq:sursol}
    \phi(x_0)\geq \inf_{  (\trajxo,\a)  }
    \apg  \int_0^\tau   \ell\big(\trajxo (t),\a (t)\big)  e^{-\lambda t} dt  +  e^{-\lambda \tau}  \phi (\trajxo(\tau))\chg\,.
    \end{equation}
   
    We use the expansion
    \be \label{exp}
    \begin{aligned}
    e^{-\lambda \tau}\phi(\trajxo (\tau))=& e^{-\lambda 0}\phi(x_0) +
    \int_0^\tau \Big\{ b_1(\trajxo(s),\alpha_1(s))\cdot D\phi(\trajxo(s)) \mathds{1}_{\Euno}(s) \\
    & + b_2(\trajxo(s),\alpha_2(s)) \cdot D\phi(\trajxo(s)) \mathds{1}_{\Edue }(s)  \\[2mm]
    &+ b_\H(\trajxo(s),\a(s))\cdot D\phi(\trajxo(s))\mathds{1}_{\Eh }(s)\\
    & - \lambda\phi(\trajxo(s))\Big\}\, e^{-\lambda s} \,ds\,.
    \end{aligned}
    \ee

    Using that $\mathds{1}_{\Euno}+\mathds{1}_{\Edue}+\mathds{1}_{\Eh}=1$ for the $(-\lambda\phi)$-term, we rewrite \eqref{ineq:sursol}
    with three contributions (for simplicity, we drop the '$s$'-dependence in the integrands and use the inversion sup/inf)
    $$\begin{aligned}
   0 \leq & \sup_{  (\trajxo,\a) } \int_0^\tau \Bigg\{
    \Big(-l_1(\trajxo,\alpha_1) - b_1(\trajxo,\alpha_1)\cdot D\phi(\trajxo) + \lambda\phi(\trajxo)\Big) \mathds{1}_{\Euno}(s)\\
    & + \Big(-l_2(\trajxo,\alpha_2) - b_2(\trajxo,\alpha_2)\cdot D\phi(\trajxo) + \lambda\phi(\trajxo)\Big) \mathds{1}_{\Edue}(s)\\
    & + \Big(-l_\H(\trajxo,\a) - b_\H(\trajxo,\a)\cdot D\phi(\trajxo) + \lambda\phi(\trajxo)\Big) \mathds{1}_{\Eh}(s)
    \Bigg\}\, e^{-\lambda s} \,ds\,.
    \end{aligned}
    $$
    Since the Hamiltonians are defined as supremum of the various quantities that appear in the integrand, we deduce that necessarily
    $$\begin{aligned}
   0 \leq & \sup_{  (\trajxo,\a) }\int_0^\tau \Bigg\{
    H_1\bigg(\trajxo(s),\phi(\trajxo(s)),D\phi(\trajxo(s))\bigg) \mathds{1}_{\Euno}(s)\\
    & + H_2\bigg(\trajxo(s),\phi(\trajxo(s)),D\phi(\trajxo(s))\bigg) \mathds{1}_{\Edue}(s)\\
    & + H_T^{*}\bigg(\trajxo(s),\phi(\trajxo(s)),D_\H\phi(\trajxo(s))\bigg)\mathds{1}_{\Eh}(s)    \Bigg\}\, e^{-\lambda s} \,ds\; ,
    \end{aligned}
    $$
    where $H_T^{*}=H_T$ if $\VF=\VFm$ and $H^{*}_T=\HTreg$ if $\VF=\VFp$.
    Next, we use that $H_1,H_2,H_T^*\leq\max(H_1,H_2)$
    together with $\mathds{1}_{\Euno}+\mathds{1}_{\Edue}+\mathds{1}_{\Eh}=1$ so that we arrive at
    \be \label{eq:super1}
    0 \leq    \sup_{  (\trajxo,\a)} \frac{1}{\tau} \int_0^\tau \max(H_1,H_2)\big(\trajxo(s),\phi(\trajxo(s)),D\phi(\trajxo(s))\big)\, e^{-\lambda s} \,ds\,.
    \ee

    Because of the regularity of $\phi$ and the continuity of the Hamiltonians we have that
    $$\max(H_1,H_2)\big(\trajxo(s),\phi(\trajxo(s)),D\phi(\trajxo(s))\big) = \max(H_1,H_2)\big(x_0,\phi(x_0),D\phi(x_0)\big)+ o(1)$$
    where  $o(1)$ denotes a quantity which tends to $0$ as  $s \rightarrow 0$, uniformly with respect to the control.  Therefore
    the sup in \eqref{eq:super1} can  be wiped out  and  sending $\tau\to0$,  we obtain
    $$\max(H_1,H_2)\big({x_0},\phi({x_0}),D\phi({x_0})\big)\geq 0$$ which means in the viscosity sense that the supersolution condition
    is verified on $\H$.

    Now we prove the subsolutions inequalities.
    We consider $\phi\in C^1(\R^N)$ and $x_0 \in \R^N$
    such that $\VF-\phi$ has a local maximum at $x_0$, that is, for some $r>0$ we have
    $$ \VF(x_0)-\phi(x_0)\geq \VF(x)-\phi(x) \quad\hbox{if   } |x-x_0|<r \,.$$
    Again, we assume that this max is zero for simplicity.

    Here also, if  $x_0\in\Omega_1$ (or $\Omega_2$), we can always find a time  $\tau$ small enough so that $|\trajxo (t)-x_0|<r$
    and $\trajxo (t)\in\Omega_1$ (or $\Omega_2$) for $0<t<\tau$. In this case the proof is classical (See \cite{BCD}, \cite{Ba}).
    So, assume that $x_0\in\H$, $\tau$ is small enough so that $|\trajxo (t)-x_0|<r$ for $t<\tau$.
    By  the Dynamic Programming Principle we have 
    \begin{equation}\label{ineq:DPP}
    \VF(x_0) = \inf_{  (\trajxo,\a) }
    \apg  \int_0^\tau   \ell\big(\trajxo (t),\a (t)\big)  e^{-\lambda t} dt  +  e^{-\lambda \tau}  \VF (\trajxo(\tau))\chg,
    \end{equation}
   thus
      \begin{equation}\label{ineq:sousol}
    \phi(x_0)\leq \inf_{  (\trajxo,\a)   }
    \apg  \int_0^\tau   \ell\big(\trajxo (t),\a (t)\big)  e^{-\lambda t} dt  +  e^{-\lambda \tau}  \phi (\trajxo(\tau))\chg.
    \end{equation}

    We distinguish now 5 sub-cases. Notice that since the inf is taken on $\Treg$ for $\VFp$, the third possibility below does not occur in this case.

    \noindent \textsc{Case 1 --} Let $\aluno, \aldue \in A$ be any constant control such that $b_1(x_0,\aluno) \cdot e_N >0$ and   $b_2(x_0,\aldue) \cdot e_N >0$.
    Then there exists  a time $\tau$ such that the controlled trajectory $(\trajxo,a)$ lives in $\Omega_1$, for all $s \in ]0,\tau]$. Therefore, by the inequality \eqref{ineq:sousol}, the expansion  \eqref{exp}  and classical arguments (dividing by $\tau$ and letting $\tau \rightarrow 0$), we obtain
    \be \label{sous1}
     -b_1(x_0,\aluno) \cdot D\phi(x_0)-l_1(x_0,\aluno)+\lambda \phi(x_0) \leq 0\,.
    \ee

    \noindent \textsc{Case 2 --} Let $\aluno, \aldue \in A$ be any constant control such that $b_1(x_0,\aluno) \cdot e_N <0$ and   $b_2(x_0,\aldue) \cdot e_N <0$.
  By the same argument as in case 1) we obtain
    \be \label{sous2}
     -b_2(x_0,\aldue) \cdot D\phi(x_0)-l_2(x_0,\aldue)+\lambda \phi(x_0) \leq 0\,.
    \ee

    \noindent \textsc{Case 3 --}  Let $\aluno, \aldue \in A$ be any constant control such that $b_1(x_0,\aluno) \cdot e_N >0$ and   $b_2(x_0,\aldue) \cdot e_N <0$
  (we can allow here also the case of one of the two to be zero).
    There exists then a trajectory $(\trajxo,a) \in \Tc_{x_0}$ such that $\trajxo(s) \in \H$ for a small time $\tau$. Indeed, if $y\in \H$ is close to $x_0$ and
    $\mu=\mu(y)$ is defined as follows  $$\mu(y):= \frac{-b_2(y,\aldue) \cdot e_N}{ (b_1(y,\aluno) - b_2(y,\aldue))\cdot e_N}, $$
    we consider the  solution of $\dot{x}(s)=\mu(x(s)) b_1(x(s), \aluno)+ (1-\mu(x(s))) b_2(x(s),\aldue)$, $x(0)=x_0$.
    By the regularity of  $b_1$ and $b_2$ (and thus of $\mu(y)$) the Cauchy-Lipschitz Theorem applies and it is easy to check that, by the definition of $\mu$, this trajectory lives in $\H$ in the interval $[0,\tau]$, for $\tau$ small enough. Moreover, by the signs of $b_1(x_0,\aluno) \cdot e_N$ and   $b_2(x_0,\aldue) \cdot e_N $, we have $0\leq \mu(x(s))\leq 1$ and therefore, $x(\cdot)$ is a controlled trajectory associated to $a=(\alpha_1,\alpha_2, \mu(x(\cdot)))$ on $[0,\tau]$. By Inequality \eqref{ineq:sousol}, the expansion  \eqref{exp}  and classical arguments, we obtain
 \be \label{sous3}
     -b_\H(x_0,\a) \cdot D\phi(x_0)-l_\H(x_0,\a)+\lambda \phi(x_0) \leq 0\,.
    \ee

    \noindent \textsc{Case 4 --}  Let $\aluno, \aldue \in A$ be any constant controls such that $b_1(x_0,\aluno) \cdot e_N <0$ and   $b_2(x_0,\aldue) \cdot e_N >0$  (here also we can allow  the case of one of the two to be zero). By the same argument as in case 3) we obtain that (\ref{sous3}) holds.\\

    \noindent \textsc{Case 5 --}  Let $\aluno, \aldue \in A$ be any constant control such that $b_1(x_0,\aluno) \cdot e_N =0$ and  $b_2(x_0,\aldue) \cdot e_N =0$.
 By the controllability assumption [H2], there exist $\alpha^-_1,\alpha^+_2$ such that $b_1(x_0,\alpha^-_1) = -\delta e_N$ and
 $b_2(x_0,\alpha^+_2)  = \delta e_N$ ($\delta >0$ given by [H2]). For $0<\eta,\eta' <1$, by the convexity assumption in   [H2], we can find $\alpha^\eta_1$ and
$\alpha_2^{\eta'}$ such that $b_1(x_0,\alpha^\eta_1)=\eta b_1(x_0,\alpha^-_1) + (1-\eta) b_1(x_0,\aluno)$  and
$b_2(x_0,\alpha^{\eta'}_2)={\eta'} b_2(x_0,\alpha^-_2) + (1-{\eta'}) b_2(x_0,\aldue)$, respectively.
Therefore $b_1(x_0,\alpha^\eta_1) \cdot e_N = -\eta\delta <0$ and $b_2(x_0,\alpha^\eta_2) \cdot e_N = {\eta'}\delta  >0$,  and, by  arguing as in case {\it 4)}, we obtain again that (\ref{sous3}) holds
with $\a=( \alpha^\eta_1, \alpha^{\eta'}_2, \bar\mu(\eta,\eta'))$ where
$$ \bar \mu(\eta,\eta'):=\frac{-b_2(x_0,\alpha^{\eta'}_2) \cdot e_N}{ (b_1(x_0,\alpha^\eta_1) - b_2(x_0,\alpha^{\eta'}_2))\cdot e_N}=\frac{\eta'}{\eta+\eta'}\; ,$$
 for all $0<\eta,\eta' <1$. By construction, $b_1(x_0,\alpha^\eta_1) \rightarrow b_1(x_0,\aluno)$ and $b_2(x_0,\alpha^{\eta'}_2) \rightarrow b_2(x_0,\aldue)$ as $\eta,\eta' \ds 0$ and, on the other hand, given some $0\leq \mu \leq 1$, we can let $\eta,\eta'$ tend to $0$ in  such a way that $ \bar\mu(\eta,\eta') \to  \mu$.
 Then, \eqref{sous3} holds for any $0\leq \mu \leq 1$ by letting  $\eta,\eta' \ds 0$ in a suitable way, recalling also the continuity of $b_1$ and $b_2 $.

  By remarking that cases 3, 4 and 5 imply that either  $-b_1(x_0,\aluno) \cdot D\phi(x_0)-l_1(x_0,\aluno)+\lambda \phi(x_0) \leq 0$ or
  $-b_2(x_0,\aldue) \cdot D\phi(x_0)-l_2(x_0,\aldue)+\lambda \phi(x_0) \leq 0$,  and that we classified all the possible constant controls
   we can conclude that
    $$
    \min(H_1,H_2)(x_0, \phi(x_0),D\phi(x_0)) \leq 0
    $$
  and the proof of \eqref{eq:HJ} is complete.

\

    Let us now prove that $\VFm$  verifies \eqref{eq:HJUm}.    We consider then $\phi \in C^1(\R^{N-1})$ and $x_0=(x'_0,0) \in \H$
    such that $x'\mapsto \VFm(x',0)-\phi (x')$ has a local maximum at $x'_0$.  We assume that this max is zero for simplicity
    and we extend the test function as follows: $\tilde\phi(x',x_N)=\phi(x')$ which is a $C^1(\R^N)$ function, independent of the $N$-th variable.
    Notice that $D\tilde\phi(x)=(D_\H\phi(x'),0)$.

    If $(\trajxo,a)$ is a controlled trajectory such that $\trajxo (t) \in \H$ for $t\in [0,\bar\tau]$ for some $\bar\tau >0$,
    we have for $0<\tau < \bar \tau$, by  the Dynamic Programming Principle
        \begin{equation*}
    \VFm(x_0)\leq
  \int_0^\tau   \ell\big(\trajxo (t),\a (t)\big)  e^{-\lambda t} dt  +  e^{-\lambda \tau}  \VFm(\trajxo(\tau))\,,
    \end{equation*}
   which implies
      \begin{equation*}
    \tilde\phi(x_0)\leq  \int_0^\tau   \ell\big(\trajxo (t),\a (t)\big)  e^{-\lambda t} dt  +  e^{-\lambda \tau}  \tilde\phi (\trajxo(\tau))\; .
    \end{equation*}
    The proof follows  the same arguments as before in the proof of $\VFm$ being a subsolution of \eqref{eq:HJ} for the cases 3,4 and 5
    (for which we have indeed $\trajxo (t) \in \H$ for $t\in [0,\bar\tau]$ for some $\bar\tau >0$).
    In other words, since we are considering only the controls in $A_0(x_0)$, we  do not have controls fulfilling cases 1 and 2.
    Therefore, all the possible controls in $A_0(x_0)$ are considered for which we obtain \eqref{sous3}.
    Thus,
    $$ \sup_{\a \in A_0(x_0)} \big\{-b_\H(x_0,\a) \cdot D\tilde\phi(x_0)-l_\H(x_0,\a)+\lambda \tilde\phi(x_0) \big\} \leq 0\,,
    $$
    that we interpret as follows:
    $$ H_T\big(x_0,\phi(x_0), D_\H\phi(x_0)\big)=
    \sup_{\a \in A_0(x_0)} \big\{-b_\H(x_0,\a) \cdot (D_\H\phi(x_0'),0)-l_\H(x_0,\a)+\lambda\phi(x_0') \big\} \leq 0\,,
    $$
    hence \eqref{eq:HJUm} holds.
    In order to prove that  $\VFp$  verifies   \eqref{eq:HJUp}   we argue exactly as before remarking that we do not have to consider cases 1, 2 and 3.
\end{proof}

%

In what follows, we are going to consider control problems set in either $\Omega_1$ or $\Omega_2$ (or their closure).
For the sake of clarity we use the following notation.
If $x_0 \in \Omega_i$, ($i=1,2$) and $\alpha_i (\cdot)\in L^\infty(\R^+;A_i)$, we will denote by   $\trajyali(\cdot)$ the solution of the
following ode
\be \label{def:trajY}
  \dottrajyali(s) = b_i(\trajyali(s), \alpha_i(s))\quad,\quad \trajyali(0) = x_0\: .
 \ee

Our next result is a (little bit unusual) supersolution property which is satisfied by $ \VFp$ on $\H$.

\begin{theo}  \label{teo:condplus.VFp}
Assume [H0], [H1] and [H2]. Let $\phi\in C^1(\R^{N-1})$ and suppose that $x_0^\prime$ is a minimum point of
$x' \mapsto \VFp (x',0)-\phi (x')$. Then we have either  \\[2mm]
{\bf A)} There exist $\eta >0$, $i=1$ or $2$ and a control $\alpha_i (\cdot)$ such that, if $x_0:=(x_0^\prime,0)$, $\trajyali(s) \in \Omega_i $ for
all $s \in ]0,\eta]$
and
\be
\VFp(x_0) \geq
   \int_0^{\eta} l_i(\trajyali(t), \alpha_i(t)) e^{-\lambda t} dt + \VFp (\trajyali(\eta))  e^{-\lambda \eta}
\ee
or \\
{\bf B)} it holds
\be  \label{condhyperSUP.VFp}
\HTreg \big(x_0,\VFp(x_0), D_\H\phi(x_0')\big) \geq 0.
\ee
\end{theo}

We skip the proof of this result to reduce the length of our paper since it is similar to the proof of Theorem~\ref{teo:condplus}.

\section{Properties of viscosity sub and supersolutions}\label{sec:Prop-sub-super}

In this section  we describe the properties fulfilled by the sub and supersolutions of system  \eqref{Bellman-Om}-\eqref{Bellman-H-sub}-\eqref{Bellman-H-sup}. We are going to consider only bounded sub and supersolutions, a natural class according to Section~\ref{sec:control}.
Because of [H2], the subsolutions are automatically Lipschitz continuous since the Hamiltonians are coercive but, a priori, the supersolutions may be  only  lower semicontinuous.

We first prove  that any subsolution of \eqref{Bellman-Om}-\eqref{Bellman-H-sub}-\eqref{Bellman-H-sup} is a viscosity subsolution of $\HTreg=0$ on the hyperplane $\H$.

\begin{theo} \label{teo:sotto}
Assume [H0], [H1] and  [H2]. If $u:\R ^N \ds \R$ is a bounded, Lipschitz continuous subsolution of \eqref{Bellman-Om}-\eqref{Bellman-H-sub}-\eqref{Bellman-H-sup}, then $u$ is a subsolution of (\ref{eq:HJUp}).
%
\end{theo}

\begin{proof} Let $\phi(\cdot)$ be a $C^1$-function on $\R^{N-1}$ and $x'_0$ a maximum point of $x' \mapsto u(x',0)-\phi(x')$,
our aim is  then to prove that, for any $\a \in \Aoreg$ where $x_0=(x'_0,0)$, we have
\be  \label{sotto:tesi1}
-b_\H(x_0,a) \cdot (D_\H\phi(x'_0),0)-l_\H(x_0,a) +\lambda u(x_0) \leq 0.
\ee
We first remark that it is sufficient to prove this inequality for the elements $\a$ of $\Aoreg$
such that
\be  \label{sotto:sottocasostretto}
b_1(x_0,\aluno) \cdot e_N <0 , \: b_2(x_0,\aldue) \cdot e_N >0  \mbox{ and } (\mu b_1(x_0,\aluno)+(1-\mu)  b_2(x_0,\aldue)) \cdot e_N =0.
\ee
Indeed, the case of non-strict inequalities can be recovered thanks to assumptions [H0]--[H2], with the same argument as in the proof of Theorem
\ref{teo:HJ}, \eqref{eq:HJUp}--\textsc{Case 5}.

We fix now any triple $(\aluno,\aldue,\mu)$ such that \eqref{sotto:sottocasostretto} is fulfilled and, as usual,
we define the function $\tilde{\phi}(x',x_N):=\phi(x')$ for which $D\tilde{\phi}(x_0)=(D_\H \phi(x'),0)$.

For $0<\eps \ll 1$,  we consider the function
\begin{equation}  \label{tra:test}
 x \mapsto u(x)-\tilde{\phi}(x) - \eta x_N -\frac{x_N^2}{\eps^2} - |x-x_0|^2:=  u(x)-\psi_\eps(x)
\end{equation}
where  the constant $\eta \in \R$ is  chosen as follows: we consider the solution $\bar{\eta} \in \R$ of
$$ -b_1(x_0,\alpha_1) \cdot  (D\tilde{\phi}(x_0) +\bar{\eta}e_N) - l_1(x_0,\aluno) + \lambda u(x_0) = 0\; .$$
Such a solution exists because of Property~(\ref{sotto:sottocasostretto}) on $b_1(x_0,\alpha_1)$ and we choose
$ \eta > \bar{\eta}$. Therefore
$$ -b_1(x_0,\alpha_1) \cdot  (D\tilde{\phi}(x_0) +\eta e_N) - l_1(x_0,\aluno) + \lambda u(x_0) >0\; .$$

By standard arguments, the function $u-\psi_\eps$ has a local maximum
$x_\eps$ in $\R^N$ and $x_\eps \ds x_0$ as $\eps \ds 0$. We want first to show that for $\eps>0$ small enough
(the other parameters being fixed for the moment), $x_\eps$ necessarily belongs to $\overline{\Omega}_2$ due to the penalization.
So, assume on the contrary that $x_\eps \in \Omega_1$. Since $u$ is a subsolution of \eqref{Bellman-Om}-\eqref{Bellman-H-sub}-\eqref{Bellman-H-sup}, we have
$$H_1(x_\eps,u(x_\eps), D\psi_\eps(x_\eps)) \leq 0 \; ,$$
which implies
\be \label{cond:omega1}
  -b_1(x_\eps,\aluno)  \cdot D\psi_\eps(x_\eps) -  l_1(x_\eps,\aluno) +\lambda u(x_\eps) \leq 0 \; .
\ee
But $\displaystyle D\psi_\eps(x_\eps)=D\tilde{\phi}(x_\eps) +\eta e_N + 2\frac{(x_\eps)_N}{\eps^2} e_N+ 2(x_\eps-x_0)$ and therefore
$$
-b_1(x_\eps,\aluno)  \cdot D\psi_\eps(x_\eps)=-b_1(x_\eps,\aluno)  \cdot (D\tilde{\phi}(x_0) +\eta e_N) -2\frac{(x_\eps)_N}{\eps^2} b_1(x_\eps,\aluno)\cdot e_N +o_\eps(1)\,,$$
where $o_\eps(1)$ is a quantity going to zero as $\eps\to0$, the other parameters being fixed.
But, because again of Property~(\ref{sotto:sottocasostretto}) and the fact that $x_\eps \in \Omega_1$ which implies $(x_\eps)_N>0$, we have
$$
-2\frac{(x_\eps)_N}{\eps^2} b_1(x_\eps,\aluno)\cdot e_N > 0\; .
$$
Finally, recalling the continuity of $b_1, l_1, \tilde{\phi}$ and $u$, we deduce that
$$ -b_1(x_\eps,\aluno)  \cdot D\psi_\eps(x_\eps) -  l_1(x_\eps,\aluno) +\lambda u(x_\eps)\geq \qquad\qquad\qquad\qquad\qquad$$
$$\qquad\qquad\qquad\qquad\qquad
-b_1(x_0,\alpha_1) \cdot  (D\tilde{\phi}(x_0) +\eta e_N) - l_1(x_0,\aluno) + \lambda u(x_0) + o_\eps(1)\; .$$
Our choice of $\eta$ implies that Inequality~(\ref{cond:omega1}) cannot hold for $\eps$ small enough,
and therefore $x_\eps \in \overline \Omega_2$.

In the same way,  if $x_\eps \in \H$, we have
 \be \label{cond:Hmin}
 \min(H_1(x_\eps,u(x_\eps), D\psi_\eps(x_\eps)),H_2(x_\eps,u(x_\eps), D\psi_\eps(x_\eps))) \leq 0 \ee
but the above proof shows that, for $\eps$ small enough,
$$H_1(x_\eps,u(x_\eps), D\psi_\eps(x_\eps))  > 0\; ,$$
and therefore
$$H_2(x_\eps,u(x_\eps), D\psi_\eps(x_\eps)) \leq 0\; .$$
In particular, this implies
\be \label{cond:Hiper}
 -b_2(x_\eps,\aldue)  \cdot D\psi_\eps(x_\eps) -  l_2(x_\eps,\aldue) +\lambda \psi_\eps(x_\eps) \leq 0  \; .
\ee
Now $x_\eps \in \overline{\Omega}_2$ which implies $(x_\eps)_N\leq 0$ and invoking again Property~(\ref{sotto:sottocasostretto}), we have
$$
-2\frac{(x_\eps)_N}{\eps^2} b_2(x_\eps,\aluno)\cdot e_N \geq 0\; ,
$$
since $b_2(x_\eps,\aluno)\cdot e_N\geq 0$ for $\eps$ small enough. This yields
$$-b_2(x_0,\alpha_2) \cdot  (D\tilde{\phi}(x_0) +\eta e_N) - l_2(x_0,\aluno) + \lambda u(x_0) + o_\eps(1)\leq 0\; .$$
In this inequality, we first let $\eps$ tend to $0$ and then $\eta$ tend to $\bar{\eta}$.

In order to conclude, we use the value of $\bar{\eta}$, namely
$$
\bar{\eta}= \frac{-b_1(x_0,\alpha_1) \cdot  D\tilde{\phi}(x_0) -l_1(x_0,\aluno) + \lambda u(x_0)}{b_1(x_0,\alpha_1) \cdot e_N } \; ,
 $$
and an easy computation on the inequality
$$-b_2(x_0,\alpha_2) \cdot  (D\tilde{\phi}(x_0) +\bar{\eta} e_N) - l_2(x_0,\aluno) + \lambda u(x_0) \leq 0\; ,$$
provides the desired inequality.
\end{proof}

\

Now we prove two properties verified by sub and supersolutions in the domains  $\Omega_i$, that will be  important to obtain the uniqueness results.

\begin{lem}  \label{reverse}
Assume [H0], [H1] and [H2]. Let $v:\R^N \ds \R$ be a lsc supersolution of  \eqref{Bellman-Om}-\eqref{Bellman-H-sub}-\eqref{Bellman-H-sup}, and
 $u:\R^N \ds \R$ be  a Lipschitz continuous subsolution of  \eqref{Bellman-Om}-\eqref{Bellman-H-sub}-\eqref{Bellman-H-sup}.
Then, if $x_0 \in \Omega_i$, we have for all $t\geq 0$
 \be \label{reverse:super}
v(x_0) \geq
 \inf_{\alpha_i(\cdot),\theta_i}
\,\biggl[\int_0^{t\wedge\theta_i} l_i(\trajyali(s),\alpha_i(s))
e^{-\lambda s}ds
+v(\trajyali(t\wedge\theta_i))e^{-\lambda( t\wedge\theta_i)}\biggr]\; ,
\ee
and
\be \label{reverse:sub}
u(x_0) \leq
 \inf_{\alpha_i(\cdot),\theta_i}
\,\biggl[\int_0^{t\wedge\theta_i} l_i(\trajyali(s),\alpha_i(s))
e^{-\lambda s}ds
+u(\trajyali(t\wedge\theta_i))e^{-\lambda( t\wedge\theta_i)}\biggr]\; ,\ee
where $\trajyali$ is the solution of the ode \eqref{def:trajY}
and the infima are taken on all stopping time $\theta_i$ such that $\trajyali(\theta_i) \in \H$ and $\tau_i \leq \theta_i \leq \bar \tau_i$ where $\tau_i$ is the exit time of the trajectory $\trajyali$ from $\Omega_i$ and $ \bar \tau_i$ is the one from $\overline{\Omega}_i $. \\
\end{lem}
\begin{proof}
For $\chi = u$ or $v$, we consider  the exit time-Dirichlet problems
\begin{equation}\label{ET-DP}
\left\{
\begin{array}{rl}
w_t + H_i(x,w,Dw)=0 &  \hbox{ in  }
 \Omega_i \times (0,+\infty) \\
w(x,0) = \chi(x) &  \hbox{ on  } \overline{\Omega}_i \\
w(x,t) = \chi (x) &  \hbox{ on  } \partial\Omega_i \times (0,+\infty)\; .
\end{array}
\right.
\end{equation}
The proofs of (\ref{reverse:super}) and (\ref{reverse:sub}) are slightly different. Property (\ref{reverse:super}) directly follows from the results of Blanc \cite{Bl1,Bl2} since $v$ is a supersolution of (\ref{ET-DP}) (with $\chi = v$) while the right hand-side is the formula for the minimal supersolution (and solution) of this problem. It is worth pointing out that (i) we do not need relaxed controls because of Assumption [H2] since, in our case, relaxed controls coincide with usual $L^\infty$ controls and (ii) the results in \cite{Bl1,Bl2} are obtained in bounded domains but they can easily be extended to unbounded domains.

For (\ref{reverse:sub}), the right-hand side of the inequality is a supersolution of (\ref{ET-DP}) (with $\chi = u$) while $u$ is a subsolution. The comparison in $\Omega_i$ of these sub and supersolution follows from the result in Barles and Perthame \cite{BP2} (see also \cite{Ba}) because $u$ is (Lipschitz) continuous. The continuity of $u$ is a key point in the comparison property.
\end{proof}

We prove now a property fulfilled by a supersolution $v$ of \eqref{Bellman-Om}-\eqref{Bellman-H-sub}-\eqref{Bellman-H-sup} which
highlights the alternative between, roughly speaking: $(i)$ there exist an optimal strategy consisting in entering one of the sets
$\Omega_1/\Omega_2$; $(ii)$ the optimal strategies consist in staying on $\H$ for a while.

\begin{theo}  \label{teo:condplus}
Assume [H0], [H1] and [H2]. Let $v:\R^N \ds \R$ be a lsc supersolution of  \eqref{Bellman-Om}-\eqref{Bellman-H-sub}-\eqref{Bellman-H-sup}.
Let $\phi\in C^1(\R^{N-1})$ and $x'_0$ be a minimum point of $x' \mapsto v(x',0)-\phi(x')$.
 Then, either  \\[2mm]
{\bf A)} there exist $\eta >0$, $i=1$ or $2$ and a sequence $x_k \in \overline{\Omega}_i \rightarrow x_0=(x_0^\prime,0)$ such that
$v(x_k) \ds v(x_0)$ and, for each $k$, there exists  a control $\alpha_i^k(\cdot)$ such that the corresponding trajectory
$\trajyxki(s) \in \overline{\Omega}_i $ for all $s \in [0,\eta]$
and
\be\label{condA}
v(x_k) \geq
   \int_0^{\eta} l_i(\trajyxki(t), \alpha^k_i(t)) e^{-\lambda t} dt + v(\trajyxki(\eta))  e^{-\lambda \eta}
\ee
or \\[2mm]
{\bf B)} it holds
\be  \label{condhyperSUP}
\HT(x_0,v(x_0),D_\H\phi(x_0^\prime)) \geq 0.
\ee
\end{theo}
\begin{proof}
We are going to prove that, if {\bf A)} does not hold, then we have {\bf B)}. To do so, we first remark that, changing $\phi (x')$ in $\phi (x') -|x' - x_0^\prime|^2$, we can assume that $x_0^\prime$ is a {\em strict} local minimum point of $v(x',0)-\phi(x')$. We then  define the function  $\tilde{\phi}(x',x_N):=\phi(x')$ in order to have  $D\tilde{\phi}(x_0)=(D_\H \phi(x'_0),0)$.

For $\eps >0$ which is devoted to tend to $0$, we consider the function
$$
  v(x',x_N)-\tilde{\phi}(x',x_N)-  \delta x_N +\frac{x_N^2}{\eps ^ 2}
$$
where $\delta \in \R$ will be chosen below. By standard arguments, this function achieves its minimum near $x_0$ and,
any sequence of such minimum points $x_\eps$  converges to $x_0=(x_0^\prime,0)$.\\
{\bf $\mathbf 1^{st}$ case.}
{\it Let us first suppose that, for all $\delta\in \R$, the minimum is attained at a point $x_\eps$ of the hyperplane $\H$.} Thus, because $x_0^\prime$ is a {\em strict} local minimum point of $v(x',0)-\phi(x')$, then
 $x^\eps=(x'_0,0)=x_0$ and, since $v$ is a supersolution of \eqref{eq:HJ}, we have
 \be   \label{def:phi}
 \varphi(\delta):= \max\{ H_1(x_0,v(x_0),D_\H\phi(x_0^\prime)+ \delta e_N),  H_2(x_0,v(x_0),D_\H\phi(x_0^\prime)+ \delta e_N)\} \geq 0
 \ee
for any $\delta \in \R$. Notice that, for the sake of simplicity of notations, we have written $D_\H\phi(x_0^\prime)+\delta e_N$ instead of $(D_\H\phi(x_0^\prime),\delta)$.

The function $\varphi(\cdot)$ defined in such a way is convex and coercive (since $H_1$, $H_2$ are convex and coercive) and,  if  $\bar{\delta}$ is a minimum point of $\varphi$, we have $0 \in \partial \varphi (\bar{\delta})$.

By a classical result (see the book of Rockafellar \cite{R}), since $\varphi$ is expressed in terms of supremum of quantities like
$$- b_1(x_0,\aluno) \cdot ( D_\H\phi(x_0^\prime)+ \delta e_N) +\lambda v(x_0)-l_1(x_0,\alpha_1)\; ,$$
$$ - b_2(x_0,\aldue) \cdot ( D_\H\phi(x_0^\prime)+ \delta e_N)  +\lambda v(x_0) -l_2(x_0,\alpha_2)\; ,$$
then any element of the subdifferential of $\varphi$, is a convex combination of the gradients of such functions (namely $- b_1(x_0,\aluno) \cdot e_N$ and $- b_2(x_0,\aldue) \cdot e_N$) for the $\aluno, \aldue$ such that
$$ \varphi(\delta) = - b_1(x_0,\aluno) \cdot ( D_\H\phi(x_0^\prime)+ \delta e_N) +\lambda v(x_0) -l_1(x_0,\alpha_1)\; ,$$
$$ \varphi(\delta)= - b_2(x_0,\aldue) \cdot ( D_\H\phi(x_0^\prime)+ \delta e_N) +\lambda v(x_0)-l_2(x_0,\alpha_2)\; ,$$
i.e. for the $\aluno, \aldue$ for which the maximum is achieved.

Taking [H2] into account and looking at the meaning of these properties at the point $\bar{\delta}$, we see that it can be reduced to:
there exists a $\mu \in [0,1]$ such that
\begin{equation*}
    \begin{aligned}
    0 & = \mu (b_1(x_0,\alpha_1) \cdot e_N)+ (1-\mu)( b_2(x_0,\alpha_2 ) \cdot e_N)\,,\quad\text{and}\\
    0 \leq  \varphi(\bar{\delta}) & = \mu \big\{ - b_1(x_0,\aluno) \cdot ( D_\H\phi(x_0^\prime)+ \bar{\delta}e_N) +\lambda v(x_0)-l_1(x_0,\alpha_1)   \big \}     \\
     & + (1-\mu)  \big\{ - b_2(x_0,\aldue) \cdot ( D_\H\phi(x_0^\prime)+ \bar{\delta}  e_N)+\lambda v(x_0) -l_2(x_0,\alpha_2)  \big\} .
    \end{aligned}
\end{equation*}
Inequality  \eqref{condhyperSUP} easily follows. \\
{\bf $\mathbf 2^{nd}$ case.} As a consequence of the arguments which are used in Case~1, we see that, if $\varphi(\bar{\delta})\geq 0$ where, as above, $\bar{\delta}$ is a global minimum point of $\varphi$, we are done. Hence we may assume $\varphi( \bar{\delta} ) <0$.

We consider the function
$$\psi_{\eps}(x):= \tilde{\phi}(x)+\bar{\delta}x_N-\frac{x_N^2}{\eps ^ 2}$$
and denote by $x_{\eps}$ a minimum point of $v-\psi_{\eps}$.

Since $\varphi( \bar{\delta} ) <0$, $x_{\eps}$ cannot be on $\H$. Therefore we  can apply Lemma~\ref{reverse} which gives, for any $t>0$
 \be \label{disveps}
v(x_{\eps}) \geq
 \inf_{\alpha_i(\cdot),\theta_i}
\,\biggl[\int_0^{t \wedge\theta_i} l_i(\trajyxei(s),\alpha_i(s))
e^{-\lambda s}ds
+v(  \trajyxei(t \wedge\theta_i))e^{-\lambda( t \wedge\theta_i)}\biggr]\; ,
\ee
where we denote by  $\trajyxei(\cdot)$  the solution of the ode    \eqref{def:trajY} starting from  $x_{\eps}$ at time $0$.

Because of [H2], the infimum in \eqref{disveps}, say for $t=1$, is attained for some $\alpha^\eps_i(\cdot)$ and $\theta^\eps_i >0$, namely
\be
v(x_{\eps}) \geq
\,\biggl[\int_0^{1\wedge\theta^\eps_i} l_i( \trajyxei(s),\alpha^\eps_i(s))
e^{-\lambda s}ds
+v(\trajyxei(1\wedge\theta^\eps_i))e^{-\lambda {(1\wedge\theta^\eps_i)}}\biggr]\;.
\ee
Moreover, recalling that we are assuming that ${\bf A)}$ does not hold, we have that $\theta^\eps_i \to 0$ as $\eps \to 0$.

But using that $x_{\eps}$ is a local minimum point of $v(x)-\psi_{\eps}(x)$ we deduce that, for $\eps$ small enough
\be  \label{dispsi}
0 \geq \int_0^{\theta^\eps_i} l_i(\trajyxei(s),\alpha_i(s)) e^{-\lambda s}ds+ v(\trajyxei(\theta^\eps_i))(e^{-\lambda \theta^\eps_i} -1) +
\psi_{\eps}( \trajyxei(\theta^\eps_i))-\psi_{\eps}(x_{\eps}).
\ee
Next we remark that, since by definition $(\trajyxei(\theta^\eps_i))_N = 0$, we can drop the quadratic term in $\psi_{\eps}$.
Indeed
$$
\frac{(\trajyxei(\theta^\eps_i))_N^2}{\eps ^ 2}-
\frac{(x_{\eps})_N^2}{\eps^ 2} \leq 0\; .$$
 Therefore, inequality \eqref{dispsi} becomes
\begin{multline}
0 \geq \int_0^{\theta^\eps_i}  l_i(\trajyxei(s),\alpha^\eps_i(s)) e^{-\lambda s}ds+ v(\trajyxei(\theta^\eps_i))(e^{-\lambda \theta^\eps_i} -1) + \\
 \int_0^{\theta^\eps_i} \biggl[   D_\H\phi(\trajyxei(s)) + \bar{\delta}   \cdot e_N   \biggr]  \cdot b_i(\trajyxei(s),\alpha^\eps_i(s))ds,
\end{multline}
thus
\begin{multline}
0 \leq \int_0^{\theta^\eps_i}  \bigg\{ - l_i(\trajyxei(s),\alpha^\eps_i(s)) e^{-\lambda s} - \biggl[   D_\H\phi(\trajyxei(s)) + \bar{\delta}   \cdot e_N   \biggr]  \cdot b_i(\trajyxei(s),\alpha^\eps_i(s))  \bigg\}ds     \\ -  v(\trajyxei(\theta^\eps_i))(e^{-\lambda \theta^\eps_i} -1)   \leq \\
 \leq \int_0^{\theta^\eps_i}   \sup_{\alpha^\eps_i \in A} \bigg\{ - l_i(\trajyxei(s),\alpha^\eps_i) e^{-\lambda s} - \biggl[   D_\H\phi(\trajyxei(s)) + \bar{\delta}   \cdot e_N   \biggr]  \cdot b_i(\trajyxei(s),\alpha^\eps_i)  \bigg\}ds  \\ -  v(\trajyxei(\theta^\eps_i))(e^{-\lambda \theta^\eps_i} -1) .
\end{multline}
If we divide now by $\theta^\eps_i$ and let $\eps$ tend to $0$, we obtain by usual arguments
$$\begin{aligned}
0 &\leq  \sup_{\alpha_i \in A} \{  - l_i(x_0,\alpha_i) - ( D_\H\phi(x_0^\prime) + \bar{\delta}   \cdot e_N   )  \cdot
b_i(x_0,\alpha_i)  \} + \lambda v(x_0) \\
 &=H_i(x_0, v(x_0), D_\H\phi(x_0^\prime) + \bar{\delta}   \cdot e_N ),
\end{aligned}$$
which is a contradiction, so that the proof is complete.
\end{proof}


\section{A uniqueness result.}\label{sec:uni}

\begin{theo}  \label{Uni-dim-N}
Assume [H0], [H1] and [H2]. Let $u$ be a bounded, Lipschitz continuous subsolution of \eqref{Bellman-Om}-\eqref{Bellman-H-sub}-\eqref{Bellman-H-sup}-\eqref{Bellman-H} and $v$ be a bounded, lsc supersolution of \eqref{Bellman-Om}-\eqref{Bellman-H-sub}-\eqref{Bellman-H-sup}. Then $u \leq v $ in $\R^N$.
\end{theo}

\begin{proof} In order to justify our strategy of proof, we point out that the usual ``doubling of variables'' method, which is very classical in viscosity solutions' theory, cannot work here since if we look at a maximum of $u(x)-v(y)-\cdots$, then $x$ and $y$ can be in two different part of the domain (either $\Omega_1$ or $\Omega_2$) and we would face two completely different and therefore useless inequalities. Therefore we have to look at a maximum of $u(x)-v(x)$ and to do so, we have to (i) regularize $u$ to make it $C^1$ at least in the $x_1, \cdots, x_{N-1}$ variables and (ii) manage to turn around the difficulty of the non-compact domain.

The regularization of $u$ relies on (almost) standard arguments. We use a sequence of mollifiers
$(\rho_{\eps})_{\eps}$ defined on $\R^{N-1}$ as follows
$$ \rho_{\eps}(x)=\frac{1}{\eps^{N-1}}\rho(\frac{x}{\eps}) \; ,$$
where
$$
\rho \in C^{\infty}(\R^{N-1}), \int_{\R^{N-1}}\rho(y)dy = 1, \text{ and } {\rm supp}\{\rho\}
=B_{\R^{N-1}}(0,1).
$$
Next we consider the function $u_{\eps}$ defined in $\R^N$ by
$$
u_{\eps} (x):= \int_{\R^{N-1}}\,
u(x'-e,x_N) \rho_{\eps}(e) de\; .
$$

A key result is the
\begin{lem}
\label{convol}
There exists a function $m:(0,+\infty) \to (0,+\infty)$ with $m(0+)=0$ such that the function $u_{\eps}-m(\eps)$ is a viscosity subsolution of
\eqref{Bellman-Om}-\eqref{Bellman-H-sub}-\eqref{Bellman-H-sup}-\eqref{Bellman-H}.
\end{lem}

We skip the proof of this result which relies on standard arguments: see the book of P.L Lions \cite{L} or Barles \& Jakobsen \cite{BJ:Rate}. It is worth pointing out that it is completely standard for (\ref{Bellman-H}) which is an equation set in $\R^{N-1}$, a little bit less classical for (\ref{Bellman-Om}). We use in a crucial way the fact that $u$ is Lipschitz continuous, as a consequence of the controlability assumption [H2] (which implies that $H_1, H_2, \HT$ are coercive Hamiltonians). Of course, $m(\eps)$ is a quantity which controls the error terms through the $\lambda u$-term.

Next we have the

\begin{lem}
\label{soussol-inf}
For $M$ large enough, the function $\psi(x):=-|x|^2 - M$ is a viscosity subsolution of
\eqref{Bellman-Om}-\eqref{Bellman-H-sub}-\eqref{Bellman-H-sup}-\eqref{Bellman-H}.
\end{lem}

Again the proof is easy since the assumptions [H0]--[H2] implies that for $H=H_1, H_2, \HT$ we have
$$ H(x,t,p) \leq C_1|p| + \lambda t + C_2\; .$$
Therefore we just have to estimate $ C_1|2x| + \lambda(-|x|^2 - M) + C_2$ and the conclusion follows easily by using the
Cauchy-Schwarz inequality for the first term.

Using these lemmas, the proof of the result is easy: for $0<\mu<1$, close to $1$, we set $u_{\eps,\mu}:=\mu(u_{\eps}-m(\eps))+(1-\mu)\psi$. Notice that, by the convexity properties of $H_1, H_2, \HT$, $u_{\eps,\mu}$ is still a subsolution of (\ref{Bellman-Om})-(\ref{Bellman-H}). Moreover $u_{\eps,\mu}(x) \to -\infty$ as $|x|\to +\infty$.

Therefore we may consider $M_{\eps,\mu}:=\max_{\R^N} (u_{\eps,\mu}(x) -v(x))$ which is achieved at some point $\bar{x}\in \R^N$ and we argue by contradiction assuming that $M_{\eps,\mu}>0$.

We first remark that, necessarily, $\bar{x} \in \H$. Otherwise, we introduce the function $u_{\eps,\mu}(x) -v(x)-|x-\bar{x}|^2$ which has a strict maximum  at $\bar{x}$ and we double the variables, i.e. we consider, for $0<\beta\ll1$,
$$ (x,y)\mapsto u_{\eps,\mu}(x) -v(y)-\frac{|x-y|^2}{\beta^2}-|x-\bar{x}|^2 \; .$$
Applying readily the classical arguments and remarking that the maximum points of this function converge to $(\bar{x},\bar{x})$, we would be led to the conclusion that $u_{\eps,\mu}(\bar{x}) \leq v(\bar{x})$ and therefore $M_{\eps,\mu}\leq0$. A contradiction.

Since $\bar{x} \in \H$, we can turn to Theorem~\ref{teo:condplus}. We point out that $u_{\eps,\mu}$ is $C^1$ with respect to the $x_1, \cdots,x_{N-1}$ variables and therefore $u_{\eps,\mu}$ is both a test-function for the $v$-inequality and it satisfies the subsolution inequality in the classical sense. Either we are in the {\bf B)} case, $ H_T(\bar{x}, v(\bar{x}), D_\H u_{\eps,\mu}(\bar{x}))\geq 0$, and we conclude with a classical comparison result
that $u_{\eps,\mu}(\bar{x}) \leq v(\bar{x})$ since $$\HT(\bar{x}, u_{\eps,\mu} (\bar{x}), D_\H u_{\eps,\mu}(\bar{x}))\leq 0\; .$$

Or we are in the case {\bf A)}.  Therefore, for any $k$, we have at
$$v(x_k) \geq
   \int_0^{\eta} l_i(\trajyxki(t), \alpha^k_i(t)) e^{-\lambda t} dt + v(\trajyxki(\eta))  e^{-\lambda \eta}
$$
 and
 $$
u_{\eps,\mu}(x_k) \leq
   \int_0^{\eta} l_i(\trajyxki(t), \alpha^k_i(t)) e^{-\lambda t} dt + u_{\eps,\mu}(\trajyxki(\eta))  e^{-\lambda \eta}\; ,$$
   where the last inequality is a consequence of Lemma~\ref{reverse}.
 Subtracting these inequalities gives
$$ u_{\eps,\mu}(x_k)-v(x_k) \leq (u_{\eps,\mu}(\trajyxki(\eta))  - v(\trajyxki(\eta)) )e^{-\lambda \eta}\leq M_{\eps,\mu}e^{-\lambda \eta}\; ,$$
and letting $k$ tends to $+\infty$ yields
$$
M_{\eps,\mu} \leq M_{\eps,\mu}e^{-\lambda \eta}\; ,$$
a contradiction which proves that $M_{\eps,\mu} \leq 0$.

This means that $u_{\eps,\mu} \leq v$ in $\R^N$ and we conclude by letting first $\mu$ tend to $1$ and then $\eps$ tend to $0$.
\end{proof}

\begin{cor}  \label{sous-max-sur-min}
Assume [H0], [H1] and [H2]. \\
(i) $\VFm$ is the minimal supersolution and solution of (\ref{Bellman-Om})-(\ref{Bellman-H-sub})-(\ref{Bellman-H-sup}).\\
(ii) $\VFp$ is the maximal subsolution and solution of (\ref{Bellman-Om})-(\ref{Bellman-H-sub})-(\ref{Bellman-H-sup}).\\
(iii) There exists a unique solution of (\ref{Bellman-Om})-(\ref{Bellman-H-sub})-(\ref{Bellman-H-sup})-(\ref{Bellman-H}) and for this problem, any subsolution is below any supersolution (Strong Comparison Result).
\end{cor}

This corollary can be read in the following way: unfortunately, we do not have a Strong Comparison Result for (\ref{Bellman-Om})-(\ref{Bellman-H-sub})-(\ref{Bellman-H-sup}) but we can identify the minimal supersolution (and solution) and  the maximal subsolution (and solution). But if we add the subsolution condition (\ref{Bellman-H}), then we recover the full Strong Comparison Result. In Section~\ref{sect:oneD}, explicit 1-D examples
show that this result is optimal and we give also sufficient conditions in 1-D which ensure that that (\ref{Bellman-Om})-(\ref{Bellman-H-sub})-(\ref{Bellman-H-sup}) and (\ref{Bellman-Om})-(\ref{Bellman-H-sub})-(\ref{Bellman-H-sup})-(\ref{Bellman-H}) are equivalent.

\begin{proof}
The proof is very short since we have just to apply Theorem~\ref{Uni-dim-N} for most of the cases.

The result (i) is an immediate consequence of the fact that $\VFm$, is a subsolution of (\ref{Bellman-Om})-(\ref{Bellman-H-sub})-(\ref{Bellman-H-sup})-(\ref{Bellman-H}) , while, for (iii), the result is exactly Theorem~\ref{Uni-dim-N}.

To prove (ii) we aim to use the same arguments as in the proof of Theorem~\ref{Uni-dim-N} with Hamiltonian $\HT$ being replaced by $\HTreg$.  \\
Let $u$ be any bounded Lipschitz continuous  subsolution of \eqref{Bellman-Om}-\eqref{Bellman-H-sub}-\eqref{Bellman-H-sup}. By the above regularization, we can suppose
that it is  a  $C^1$ function  on $R^{N-1}$.  Let $\bar{x}$ be a maximum point of $u-\VFp$ on $\R^N$, and observe that,
as above, we can assume that $\bar{x} \in \H$.
We observe now that  the function  $\VFp$
fulfills  Theorem \ref{teo:condplus.VFp} at point $\bar{x}$, mimimum of  the function $x' \mapsto \VFp(x',0)-u(x',0)$.
(Note that this is the analogue of Theorem~\ref{teo:condplus}). Therefore,  we have case ${\bf A)}$
or case  {\bf B)}.  \\
If we are in case {\bf B)}, it holds
  $$
  \HTreg \big(\bar{x},\VFp(\bar{x}), D_\H u(\bar{x})\big) \geq 0,
$$
therefore, the conclusion easily follows observing that  by Theorem \ref{teo:sotto}, $u$ fulfills in the classical sense
$$
 \HTreg \big(\bar{x},u(\bar{x}), D_\H u(\bar{x})\big) \leq 0.
$$
If we are in case ${\bf A)}$, there exist $\eta >0$, $i=1$ or $2$ and a control $\alpha_i (\cdot)$ such that, $\trajyxbari(s) \in \Omega_i $ for
all $s \in ]0,\eta]$
and
$$
\VFp(\bar{x}) \geq
   \int_0^{\eta} l_i(\trajyxbari(t), \alpha_i(t)) e^{-\lambda t} dt + \VFp (\trajyxbari(\eta))  e^{-\lambda \eta}.
$$
Therefore the conclusion follows by applying \eqref{reverse:sub} in Lemma~\ref{reverse}, noticing that, since $u$ is Lipschitz continuous, this property can be extended to points of $\overline\Omega_i$ (and not only to points of $\Omega_i$).  \\
\end{proof}

\begin{rem}
The last part of the proof of Corollary~\ref{sous-max-sur-min} clearly emphasizes the uniqueness problem  with (\ref{Bellman-Om})-(\ref{Bellman-H-sub})-(\ref{Bellman-H-sup}): either, on $\H$, we allow the singular strategies
$$ (\mu b_1(x,\aluno)+(1-\mu)b_2(x,\aldue))\cdot e_N=0,\;  b_1(x,\aluno) \cdot e_N > 0,\; b_2(x,\aldue)\cdot e_N< 0 \; ,$$
and, since they are not encoded in (\ref{Bellman-Om})-(\ref{Bellman-H-sub})-(\ref{Bellman-H-sup}), we have to add (\ref{Bellman-H}) to get the uniqueness. Or we do not allow them (or they are not optimal) and we obtain the uniqueness. In any case, the choice of  the inequality $\HT$ (to be imposed) or $\HTreg$ (consequence of (\ref{Bellman-Om})-(\ref{Bellman-H-sub})-(\ref{Bellman-H-sup})) makes the difference.
\end{rem}


\section{The 1-D case}\label{sect:oneD}

In this section we go a little bit further by providing a complete classification of the value functions $\VFm$ and $\VFp$.
Moreover, we derive explicit examples highlighting the different strategie that we call ``state constraints'', ``push-push'' and ``pull-pull'' strategies, and the non-uniqueness phenomenon.

\subsection{Structure of solutions}

In order to describe the structure of solutions, we introduce the state constraint solutions $\VF_\mathrm{SC1}$ in $\Omega_1$ and $\VF_\mathrm{SC2}$ in $\Omega_2$ which are defined, for $i=1,2$ in the following way
$$
\VF_\mathrm{SCi}(x_0) :=
 \inf_{\A_\mathrm{SCi}}
\,\biggl[\int_0^{+\infty} l_i(\trajyali(s),\alpha_i(s))
e^{-\lambda s}ds\biggr]\; ,
$$
where $\A_\mathrm{SCi}$ is the set of controls $\alpha_i(\cdot)$ for which $\trajyali(s)\in \overline \Omega_i$ for any $s\geq 0$.
Note that by  the classical results  in \cite{Son}, $\VF_\mathrm{SCi}$ are solutions of $H_i=0$ on $\Omega_i$ and $H_i \geq 0$ on
$\overline \Omega_i$  (see also \cite{BCD}).

We also denote by
$$u_\H(0)=\frac{1}{\lambda}\min_{(\mu,\alpha_1,\alpha_2)\in A_0(0)}\Big\{\mu l_1(0,\alpha_1)+(1-\mu) l_2(0,\alpha_2)\Big\}\,,$$
which is the solution of $H_T(x,u,D_\H u)=0$ on $\H=\{0\}$ in this particular case. Similarly, $u_\H^\mathrm{reg}(0)$ is the same quantity
when the min is taken over the regular controls, $A_0^\mathrm{reg}(0)$.

\begin{theo}\label{thm:structure}
Assume [H0], [H1] and [H2].  The following holds $\VFm(0)=\min\big\{u_\H(0),\VF_\mathrm{SC1}(0),\VF_\mathrm{SC2}(0)\big\}$. Therefore we have
    \begin{itemize}
      \item[(i)] if the min is given by $u_\H(0)$, then $\VFm$ is the unique solution of the Dirichlet problems
      $$\begin{cases} H_1(x,w,Dw)=0 \text{ in } \Omega_1 \\
        w(0)=u_\H(0) \end{cases}\quad\text{and}\quad
        \begin{cases} H_2(x,w,Dw)=0 \text{ in } \Omega_2 \\
        w(0)=u_\H(0) \end{cases}
      $$
      \item[(ii)] if the min is given by $\VF_\mathrm{SC1}(0)$, then $\VFm\equiv\VF_\mathrm{SC1}$ in $\Omega_1$, and in $\Omega_2$
      it is the unique solution of the Dirichlet problem $H_2=0$ with boundary value $\VFm(0)=\VF_\mathrm{SC1}(0)$;

     \item[(ii)] if the min is given by $\VF_\mathrm{SC2}(0)$, then $\VFm\equiv\VF_\mathrm{SC2}$ in $\Omega_2$, and in $\Omega_1$
      it is the unique solution of the Dirichlet problem $H_1=0$, with boundary value $\VFm(0)=\VF_\mathrm{SC2}(0)$.
    \end{itemize}
    Similarly, $\VFp(0)=\min\big\{u_\H^\mathrm{reg}(0),\VF_\mathrm{SC1}(0),\VF_\mathrm{SC2}(0)\big\}$
    and the value of the min identifies $\VFp$ as above, among the three possibilities.
\end{theo}

Concerning uniqueness, a direct consequence is the following:
\begin{cor}\label{cor:uniq.dim1}
Assume [H0], [H1] and [H2] and one of the following condition holds: $(i)$ $u_\H(0)=u_\H^\mathrm{reg}(0)$;
    $(ii)$ $u_\H(0)\geq\min\big\{\VF_\mathrm{SC1}(0),\VF_\mathrm{SC2}(0)\big\}$.
    Then $\VFm\equiv\VFp$: there is a unique solution of the problem.
\end{cor}

Theorem \ref{thm:structure} follows from the conjunction of several results. We begin with the following proposition concerning  supersolutions.
\begin{pro}\label{prop:dim1.sursol}
Assume [H0], [H1] and [H2]. Let $v$ be a bounded, lsc supersolution of (\ref{Bellman-Om})-(\ref{Bellman-H-sub})-(\ref{Bellman-H-sup}).
Then $$v(0)\geq \min\big\{u_\H(0), \VF_{\rm SC1}(0), \VF_{\rm SC2}(0)\big\}=\VFm(0)\,.$$
\end{pro}

Before proving this proposition, we go back to the solution introduced in Lemma~\ref{reverse} which is considered here for $x_0\in\Omega_1$
\begin{equation}\label{ineqvw}
v(x_0) \geq w_-(x_0,t):= \inf_{\alpha_1(\cdot),\theta_1}
\,\biggl[\int_0^{t\wedge\theta_1} l_1(\trajyaluno(s),\alpha_1(s))
e^{-\lambda s}\,ds +v(\trajyaluno(t\wedge\theta_1))e^{-\lambda( t\wedge\theta_1)}\biggr]\;,
\end{equation}
where we recall that $\trajyaluno$ is the solution of the ode (\ref{def:trajY}) with $i=1$
and the infimum is taken on all stopping times $\theta_1$ such that $\trajyaluno(\theta_1) = 0$
and $\tau_1 \leq \theta_1 \leq \bar \tau_1$ where $\tau_1$ is the exit time of the trajectory
$\trajyaluno$ from $\Omega_1$ and $ \bar\tau_1$ is the one from $\overline{\Omega}_1$.

We again point out that we do not need relaxed controls in the expression of $w_-$
because of Assumption [H2] since in our case, relaxed controls coincide with usual $L^\infty$ controls.
Using this remark, we notice that, for any $x_0\neq 0$ and $t>0$, the above infimum
for $w_-$ is achieved for some control $\bar\alpha_1(\cdot)$ and some stopping time $0<{\bar \theta}_1 \leq +\infty$, namely
\begin{equation}\label{w-moins}
w_-(x_0,t) = \int_0^{t\wedge {\bar \theta}_1} l_1(\trajyaluno(s),\bar\alpha_1(s))
e^{-\lambda s}ds+v(\trajyaluno(t\wedge {\bar \theta}_1 ))e^{-\lambda (t\wedge{\bar \theta}_1)}\; .
\end{equation}

A priori, $\bar\alpha_1(\cdot)$ and ${\bar \theta}_1$ depend on $x_0$ and $t$ but we drop most of the time
this dependence for the sake of simplicity of notations.

The following lemma holds
\begin{lem}\label{monotone-y} Either $v(x_0) \geq  \VF_{\rm SC1}(x_0)$ or, for $t$ large enough, the above defined control $\bar\alpha_1(\cdot)$ and stopping time ${\bar \theta}_1$ can be chosen as  being $t$-independent, ${\bar \theta}_1$ being finite. Moreover the associated trajectory $\trajyaluno(\cdot)$ is decreasing.
\end{lem}

\begin{proof}
We assume that $v(x_0)  <  \VF_{\rm SC1}(x_0)$; the aim of the proof is to show that the second case is true.

We first remark that, necessarily, this implies that the stopping time ${\bar \theta}_1$ (which a priori depends on $t$ at this stage of the proof) remains uniformly bounded as $t\to +\infty$. Otherwise, at least up to some subsequence, we can pass to the limit in \eqref{w-moins}, using the compactness of controlled trajectories\footnote{We recall again that we do not need relaxed controls because of Assumption [H2].} and we get, by standard arguments, that 
$$ \limsup_{t\to +\infty}\,w_-(x_0,t)\geq  \inf_{\A_\mathrm{SC1}}
\,\biggl[\int_0^{+\infty} l_1(\trajyaluno(s),\alpha_1(s))
e^{-\lambda s}\,ds\biggr]=\VF_{\rm SC1}(x_0)>v(x_0)\; ,$$ which contradicts \eqref{ineqvw}.

Therefore we may assume that ${\bar \theta}_1$ remains bounded and take $t$ large enough so that
$t\wedge {\bar \theta}_1 = {\bar \theta}_1$. Then
\begin{align*}
w_-(x_0,t) &= \int_0^{\bar\theta_1} l_1(\trajyaluno(s),\bar\alpha_1(s))e^{-\lambda s}ds
+v(\trajyaluno({\bar \theta}_1))e^{-\lambda {\bar \theta}_1}\; ,\\
& = \inf_{\alpha_1(\cdot),\theta_1}
\,\biggl[\int_0^{\theta_1} l_1(\trajyaluno(s),\alpha_1(s))
e^{-\lambda s}\,ds +v(\trajyaluno(\theta_1))e^{-\lambda\theta_1}\biggr]\;,
\end{align*}
the last equality being a consequence of the optimality of the control $\bar\alpha_1(\cdot)$ and the stopping time $\bar\theta_1$. For this infinite horizon, exit time control problem, there exists an optimal control and a stopping time that we still denote by $\bar\alpha_1(\cdot)$ and $\bar\theta_1$ (which are obviously independent of $t$).

Moreover, pick any $0<\sb<\bar\theta_1$,  by the Dynamic Programming Principle for $w_-$
\begin{eqnarray*}
w_-(x_0,t) &=& \inf_{\alpha_1(\cdot),{\theta_1}}
\,\biggl[\int_0^{( t-\sb) \wedge  \theta_1 } l_1(\trajyaluno (s),\alpha_1(s))
e^{-\lambda s}ds
+w_-(\trajyaluno(( t-\sb)\wedge  \theta_1), t-\sb)e^{-\lambda (( t-\sb)\wedge  \theta_1)}\biggr]\; \\
& \leq &  \int_0^{\sb} l_1(\trajyaluno (s),\bar\alpha_1(s))
e^{-\lambda s}ds + w_-(\trajyaluno(\sb), t-\sb)e^{-\lambda \sb}\; .
\end{eqnarray*}
We deduce from this property that
$$
w_-(\trajyaluno(\sb), t-\sb)e^{-\lambda \sb}= \int_\sb^{{\bar \theta}_1} l_1(\trajyaluno (s),\bar\alpha_1(s))
e^{-\lambda s}ds
+v(\trajyaluno({\bar \theta}_1))e^{-\lambda {\bar \theta}_1}\; .$$
Therefore $w_-(\trajyaluno(\sb), t-\sb)$ is independent of $t$ as well and we drop this dependence
by just writing $w_-(\trajyaluno(\sb))$ for $0\leq \sb <{\bar \theta}_1$.

If $s \mapsto \trajyaluno(s)$ is not monotone on $[0,\bar\theta_1[$, there exists $0\leq s_1 <s_2<{\bar \theta}_1$
such that $\trajyaluno (s_1)=\trajyaluno (s_2)$. By the above property, we have
$$
w_-(\trajyaluno(s_1))e^{-\lambda s_1}= \int_{s_1}^{s_2} l_1(\trajyaluno (s),\bar\alpha_1(s))
e^{-\lambda s}ds+w_-(\trajyaluno(s_2))e^{-\lambda s_2}\; .$$
Using the fact that $\trajyaluno(s_1)=\trajyaluno(s_2)$ and the Dynamic Programming Principle which can be written as
$$
w_-(x_0,t) = \int_0^{s_1} l_1(\trajyaluno(s),\bar\alpha_1(s))
e^{-\lambda s}ds + w_-(\trajyaluno(s_1))e^{-\lambda s_1}\; ,$$
this means that, iterating the loop, from $s_1$ to $s_2$,
we have an optimal control defined for all $s>0$. More precisely we introduce the control
$$ \tilde \alpha_1 (s) = \left\{
\begin{array}{rl}
\bar\alpha_1(s) & \hbox{if  }0\leq s \leq s_2\; ,\\
\bar\alpha_1(s-(s_2-s_1)) & \hbox{if  }s_2 \leq s \leq s_2+(s_2-s_1)\; ,\\
&   \vdots \\
\bar\alpha_1(s-k(s_2-s_1)) & \hbox{if  }s_2 + k(s_2-s_1)\leq s \leq s_2+(k+1)(s_2-s_1)\;. \\
\end{array}
\right.
$$
Arguing by induction, it is easy to show that the associated trajectory $\tiltrajyaluno$
remains in $\Omega_1$ for all $s>0$, that $\tiltrajyaluno(s_2+(k+1)(s_2-s_1))= \tiltrajyaluno (s_1)=\tiltrajyaluno(s_2)$ and
$$
w_-(x_0,t) = \int_0^{s_2+(k+1)(s_2-s_1)} l_1(\tiltrajyaluno(s),\tilde \alpha_1(s))
e^{-\lambda s}ds + w_-(\tiltrajyaluno(s_2))e^{-\lambda (s_2+(k+1)(s_2-s_1))}\; .$$
Letting $k \to +\infty$ gives $w_-(x_0,t) \geq \VF_{\rm SC1}(x_0)$ but the definition of
$w_-$ implies that it is in fact an equality.

Now if the trajectory is monotone, there are two possibilities: increasing or decreasing.
We remark that if $s\mapsto \trajyaluno(s)$ is increasing in $\Omega_1$,
then ${\bar \theta}_1$ is necessarily $+\infty$ and $\lim_{t \ds \infty} w_-(x_0,t) = \VF_{\rm SC1}(x_0)$.
The only remaining possibility is that the trajectory is decreasing, which ends the proof.
\end{proof}

Now we can proceed with the proof of Proposition \ref{prop:dim1.sursol}.

{\it Proof of Proposition \ref{prop:dim1.sursol}.}
We first examine the case when $v(0)<\liminf_{x\to0}v(x)$ and apply Theorem~\ref{teo:condplus}. We first remark that, if \textbf{A)} holds for some $\eta >0$ then it is also true for any $\bar\eta <\eta$, by the Dynamic Programming Principle. Therefore we can assume that $\eta$ is as small as we want.

Then the property $v(0)<\liminf_{x\to0}v(x)$ implies that, in \eqref{condA}, necessarily $x_k=0$ and $\trajyxki(\eta)=0$ if $\eta$ is small enough. A simple computation then yields $\lambda v(0) \geq \inf_{\alpha_i}\,l_i(0,\alpha_i) \geq \lambda u_\H(0)$.

On the contrary, if \textbf{B)} holds,
we know that, for any test function, the minimum is attained at $x=0$, which implies directly
$$H_T(0,v(0),D_\H v(0))=\lambda(v(0)-u_\H(0))\geq 0\,.$$
Hence, in both cases, $v(0)\geq u_\H(0)$ which implies the result.

Since $v$ is lsc, we may now assume in the rest of the proof that $v(0)=\liminf_{x\to0}v(x)=\lim_{x_k \ds 0} v(x_k)$.
The alternative of Theorem~\ref{teo:condplus} can also be applied. In case $\mathbf{B)}$, we have exactly as above:
$v(0)\geq u_\H(0)\geq \min\big\{u_\H(0), \VF_{\rm SC1}(0), \VF_{\rm SC2}(0)\big\}$.

It remains to treat case $\mathbf{A)}$. For simplicity we assume that, for any $k$, $x_k>0$, that is we always have $i=1$ and we use Lemma~\ref{monotone-y}: up to the extraction of subsequences, the following holds

\noindent $(i)$ either, for any $x_k$, we have $v(x_k)\geq \VF_{\rm SC1}(x_k)$ and
we conclude by letting $x_k\to 0$, which gives directly $v(0)\geq\VF_{\rm SC1}(0)$, since  $\VF_{\rm SC1}$ is a  continuous function thanks to assumptions [H1] and [H2];\\
$(ii)$ or for any $x_k$, the associated trajectory $Y^1_{x_k}$ is decreasing.
In this case
we notice that $0\leq Y^1_{x_k}({\bar \theta}_1)\leq x_k$ and pass to the limit in the expression
$$v(x_k)\geq w_-(x_k,t) = \int_0^{{\bar \theta}_1} l_1(Y^1_{x_k}(s),\bar\alpha_1(s))e^{-\lambda s}ds
+v(Y^1_{x_k}({\bar \theta}_1))e^{-\lambda {\bar \theta}_1}\;,
$$
which yields
$$v(0)\geq \int_0^{{\bar \theta}_1} l_1(0,\bar\alpha_1(s))e^{-\lambda s}ds
+v(0)e^{-\lambda {\bar \theta}_1}\; .
$$
Iterating the estimate as above we obtain that $v(0)\geq\VF_{\rm SC1}(0)$.

Of course, if the sequence $x_k$ lies in $\Omega_2$, we obtain in each case that $v(0)\geq\VF_{\rm SC2}(0)$.
Combining all the cases above, we get indeed
$$v(0)\geq\min\big\{u_\H(0), \VF_{\rm SC1}(0), \VF_{\rm SC2}(0)\big\}   \,.$$

It remains to proove that this min is actually $\VFm(0)$. First,
by comparison in $\Omega_i$ we get $\VFm(x)\leq \VF_{\mathrm{SC}i}(x)$, $i=1,2$.
Moreover, \eqref{eq:HJUm} gives $\VFm(0)\leq u_\H(0)$ so that we obtain
$\VFm(0)\leq\min\big\{u_\H(0), \VF_{\rm SC1}(0), \VF_{\rm SC2}(0)\big\}$. Then, since $\VFm$
is a supersolution, we can use the reverse inequality for $v=\VFm$ and conclude that equality holds.
\hfill $\Box$

Concerning subsolutions, we have
\begin{pro}\label{pro:dim1.subsol}
Assume [H0], [H1] and [H2]. Let  $u$ be a bounded, usc subsolution of (\ref{Bellman-Om})-(\ref{Bellman-H-sub})-(\ref{Bellman-H-sup}).
Then $$u(0)\leq \min\big\{u_\H^\mathrm{reg}(0), \VF_{\rm SC1}(0), \VF_{\rm SC2}(0)\big\}=\VFp(0)\,.$$
\end{pro}

\begin{proof}
    We first show the inequality on the left.
    Notice first that by \eqref{eq:HJUp}, $u(0)\leq \VFp(0)\leq u_\H^\mathrm{reg}(0)$. Also, by classical comparison arguments,
    it is clear that in $\Omega_1$, $u(x)\leq\VF_\mathrm{SC1}(x)$, and the same in $\Omega_2$ with $\VF_\mathrm{SC2}$.
    Hence the inequality indeed holds. Notice that $\VFp(0)$ itself also satisfies the inequality on the left.

    Now, for the equality on the right the argument is similar to the one we used for the supersolutions.
    Alternative $\mathbf{B)}$ of Theorem \ref{teo:condplus.VFp} translates directly into $\VFp(0)\geq u_\H^ \mathrm{reg}(0)$.
    In the other case, we use Lemma \ref{monotone-y} and conclude as above that $\VFp(0)\geq \VF_\mathrm{SC1}(0)$.
    Hence $\VFp(0)\geq\min\big\{u_\H^\mathrm{reg}(0), \VF_{\rm SC1}(0), \VF_{\rm SC2}(0)\big\}  $, and combining with the first part of this proof we get equality.
\end{proof}

We then turn to the {\bf  proof of Theorem \ref{thm:structure}}.
\begin{proof}
    We consider the case of $\VFm$, the argument being the same for $\VFp$. The first part of the Theorem has been proved in
    Proposition~\ref{prop:dim1.sursol}. Then, we solve separately the Dirichlet problems in $\Omega_1$ and $\Omega_2$, putting the
    value of the min as boundary condition at $x=0$.
    We get a solution $u_\sharp$ in $\Omega_1\cup\Omega_2$ which satisfies $u_\sharp(0)=\VFm(0)$. Hence, by uniqueness for the
    Dirichlet problem in each $\Omega_i$, we end up with $u_\sharp\equiv\VFm$.
\end{proof}

\subsection{State constraint strategies}\label{subsec:sc}

We consider dynamics which are given by
$$\dottrajxo(t)=\alpha_1(t)\text{ in }\Omega_1\,,\quad\dottrajxo(t)=\alpha_2(t)\text{ in }\Omega_2\,,$$
where $\alpha_1(\cdot),\alpha_2(\cdot)\in L^\infty\big(0,+\infty;[-1,1]\big)$ are the controls.
We are thus in the case where the dynamic reduces to $b_i(x,\alpha_i)=\alpha_i$: the control is actually
the velocity of the trajectory. Then we consider the following costs
$$l_1(x,\alpha_1)=1-\alpha_1+e^{-|x|}\text{ in }\Omega_1\,,\quad l_2(x,\alpha_2)=1+\alpha_2+e^{-|x|}\text{ in }\Omega_2\,.$$
In this case, it is rather obvious that the best strategy for $x_0>0$ consists in choosing $\alpha_1\equiv1$, which yields a
state constraint solution. To be more precise, let us first consider
$$\lambda u_\H(0)=\min_{\mu,\alpha_1,\alpha_2}\Big\{\mu(2-\alpha_1)+(1-\mu)(2+\alpha_2) : \mu \alpha_1+(1-\mu)\alpha_2=0\Big\}\,.$$
Using the compatibility condition, we can compute the minimum which is atteind for
$(\mu,\alpha_1,\alpha_2)=(1/2,1,-1)$ and gives $u_\H(0)=1/\lambda$.

Now, we can compute the state constraint solution $\VF_\mathrm{USC1}$: for $x_0>0$ we have
$$\begin{aligned}
    \VF_\mathrm{USC1}(x_0) &= \inf_{\alpha_1(\cdot)}\int_0^{+\infty}\Big(1-\alpha_1+e^{-\trajxo(t)}\Big)\,e^{-\lambda t}\,dt\\
    &= \int_0^{+\infty} e^{-\trajxo(t)}\,e^{-\lambda t}\,dt\\
    &= \int_0^{+\infty} e^{-x_0-t}\,e^{-\lambda t}\,dt\\
    &= \frac{e^{-x_0}}{1+\lambda}\,.
\end{aligned}$$
Indeed, the inf is clearly obtained for the choice $\alpha_1\equiv 1$ which implies that $\trajxo(t)=x_0+t$.
Hence $\VF_\mathrm{USC1}(0)=1/(1+\lambda)<1/\lambda=u_\H(0)$ for any $\lambda>0$. 
For symmetry reasons, we have also $\VF_\mathrm{SC2}(x)=\VF_\mathrm{SC1}(-x)$ for $x\leq0$. Then Corollary \ref{cor:uniq.dim1}
implies uniqueness, so that we conclude that for any $x\in\R$,
$$\VFm(x)=\VFp(x)=\frac{e^{-|x|}}{1+\lambda}\,.$$
For $x>0$ we can also compute the Hamiltonian associated:
$$\begin{aligned}
    H_1(x,u,p) &= \sup_{\alpha_1 \in [-1;1]}\big\{ -\alpha_1 p + \lambda u -(1-\alpha_1+e^{-|x|}) \big\}\\
    &= \sup_{\alpha_1 \in [-1;1]}\big\{ -\alpha_1(p-1)\big\} + \lambda u - 1- e^{-|x|} \\
    &= |p-1| +\lambda u - e^{-|x|} -1\,,
\end{aligned}$$
and a direct computation shows that $\VFm$ is indeed a solution of the HJB equation $H_1=0$ in $\Omega_1$.
Of course a similar calculus can be done in $\Omega_2$ which gives $H_2(x,u,p)=|p+1| +\lambda u - e^{-|x|} -1$, for 
$x<0$.

\subsection{``Push-push'' strategies}\label{subsec:push}

We now provide an example where the state constraint solutions are not necessarily the best ones.
Consider
$$\dottrajxo(t)=\alpha_1(t)\text{ in }\Omega_1\,,\quad\dottrajxo(t)=\alpha_2(t)\text{ in }\Omega_2\,,$$
where $\alpha_1(\cdot),\alpha_2(\cdot)\in L^\infty\big(0,+\infty;[-1,1]\big)$ are the controls
and the following costs
$$l_1(x,\alpha_1)=1+\alpha_1\text{ in }\Omega_1\,,\quad l_2(x,\alpha_2)=1-\alpha_2\text{ in }\Omega_2\,,$$
Considering the quantity $\min\big\{u_\H(0),\VF_\mathrm{SC1}(0),\VF_\mathrm{SC2}(0)\big\}$, it is clear that
$$\lambda u_\H(0)=\min_{\mu,\alpha_1,\alpha_2}\Big\{\mu(1+\alpha_1)+(1-\mu)(1-\alpha_2) : \mu \alpha_1+(1-\mu)\alpha_2=0\Big\}=0\,,$$
which is attained for the choice $(\mu,\alpha_1,\alpha_2)=(1/2,-1,+1)$. This corresponds to a ``push-push'' strategy, while
the state constraint solutions cannot reach the min. Indeed,
$$\begin{aligned}
    \VF_\mathrm{SC1}(x_0) & = \inf_{\alpha_1(\cdot)}\int_0^{+\infty}\big( 1+\alpha_1(t)\big)\,e^ {-\lambda t}\,dt\\
    & = \inf_{\alpha_1(\cdot)}\int_0^{+\infty}\big( 1+\dottrajxo(t)\big)\,e^ {-\lambda t}\,dt\\
    & = \inf_{\alpha_1(\cdot)}\Big(\frac{1}{\lambda} + \Big[\trajxo(t)\,e^{-\lambda t}\Big]_0^{+\infty}+
        \lambda\int_0^{+\infty} \trajxo(t)\,e^{-\lambda t}\,dt\Big)\\
    & \geq \frac{1}{\lambda} - x_0
\end{aligned}$$
which implies $\VF_\mathrm{SC1}(0)\geq 1/\lambda>0$, which is attained for $\alpha_1\equiv0$ for $x_0=0$ (the computation is similar for $\VF_\mathrm{SC2}$).

Actually, we can compute explicitly $\VFm$. We first remark that since $l_i \geq 0$, we have $\VFp \geq 0$ and $\VFm \geq 0$.
 For $x\in\Omega_1$, we choose the control
$(\mu,\alpha_1,\alpha_2)=(1/2,-1,+1)$ for $t\geq0$. Of course, $\mu$ and $\alpha_2$ are not relevant before the first hitting time $\tau(x)=|x|$.
This strategy consists in reaching $\H=\{x=0\}$ as fast as possible, then to stay in $\H$ using a ``push-push'' strategy.
With this particular choice,
$$\VFm(x)\leq\int_0^{x}(1+\alpha_1(t))\,e^{-\lambda t}\,dt + \int_{x}^{+\infty} l_\H\big(x(t), a(t)\big)\,e^{-\lambda t}\,dt=0\,,$$
so that $\VFm(x)\leq0$. Hence we deduce that
$\VFm\equiv0$, since the computation for $x<0$ is similar.
Of course this is an obvious solution of the associated HJB equation which reads in this case (for $x>0$)
$$|u_x+1|+\lambda u =1\,.$$
Since the regular strategies give us $\VFm(x)=0$ and $\VFp(x)\geq 0$  the non singular strategies cannot  be better here,
therefore $u_\H(0)=u_\H^\mathrm{reg}(0)$, which implies that uniqueness holds -- see Corollary \ref{cor:uniq.dim1}: $\VFp\equiv\VFm\equiv0$ is
actually the unique solution of \eqref{Bellman-Om}.

\subsection{Non-uniqueness and ``pull-pull strategies''}\label{subsec:pull}

Next we consider the ``converse situation'', with the same dynamics but now with the costs
$$l_1(x,\alpha_1)=1-\alpha_1+|x|\text{ in }\Omega_1\,,\quad l_2(x,\alpha_2)=1+\alpha_2+|x|\text{ in }\Omega_2\,,$$
where $\alpha_1(\cdot),\alpha_2(\cdot)\in L^\infty\big(0,+\infty;[-1,1]\big)$.
Of course, in this example the running cost is not bounded because of the $|x|$-term but a slight modification
would give a similar result. We keep it as it is to make simple computations.
We have first
$$\lambda u_\H(0)=\min_{\mu,\alpha_1,\alpha_2}\Big\{\mu(1-\alpha_1+|0|)+(1-\mu)(1+\alpha_2+|0|) : \mu \alpha_1+(1-\mu)\alpha_2=0\Big\}=0\,,$$
which is attained this time for the ``pull-pull'' strategy $(\mu,\alpha_1,\alpha_2)=(1/2,1,-1)$.

Notice that if now we consider only the  regular trajectories, the "pull-pull" strategies are forbidden. In this case, the restrictions $\alpha_1 \leq 0,\ \alpha_2 \geq 0$ imply
$$\mu(1-\alpha_1)+(1-\mu)(1+\alpha_2)=1-\mu\alpha_1+(1-\mu)\alpha_2  \geq 1 \;. $$
Hence we obtain $u_\H^\mathrm{reg}(0)=1/\lambda$, which is attained for $\alpha_1=\alpha_2=0$.

We  compute now  the state constraint solutions
$$\begin{aligned}
    \VF_\mathrm{SC1}(x_0)&=\inf_{\alpha_1(\cdot)} \int_0^{+\infty} \bigg(1-\alpha_1(t)+\trajxo(t)\bigg)e^{-\lambda t}\,dt\\
    &=\frac{1}{\lambda}+\inf_{\alpha_1(\cdot)} \left\{
    \int_0^{+\infty} -\alpha_1(t)\,e^{-\lambda t}\,dt + \Big[\trajxo(t)\frac{e^{-\lambda t}}{-\lambda}\Big]_0^{+\infty}
        +\frac{1}{\lambda} \int_0^{+\infty}\dottrajxo(t)\,e^{-\lambda t}\,dt\right\}\\
    &=\frac{1+x_0}{\lambda}+\inf_{\alpha_1(\cdot)}\Bigg\{\Big(\frac{1}{\lambda}-1\Big)\int_0^{+\infty}\alpha_1(t)e^{-\lambda t}\,dt\Bigg\}\,.
\end{aligned}$$
For instance, if $\lambda=1$ we get the simple solution $\VF_\mathrm{SC1}(x)=x+1.$ More generally,
if $\lambda>1$, the best strategy consists in choosing $\alpha_1\equiv 1$
and the computation gives
$$\VF_\mathrm{SC1}(x)=\frac{x}{\lambda}+\frac{1}{\lambda^2}\,.$$
For symmetry reasons (or a direct calculation) we have $\VF_\mathrm{SC2}(x)=\VF_\mathrm{SC1}(-x)$ for $x\leq0$.

If $0<\lambda<1$ then we compute as follows, using the hitting time $\tau(x_0)$ and the ``pull-pull'' strategy for $t>\tau(x_0)$
$$\begin{aligned}
\VF_\mathrm{SC1}(x_0)&=\inf_{\alpha_1(\cdot)} \int_0^{+\infty} \bigg(1-\dottrajxo(t)+\trajxo(t)\bigg)e^{-\lambda t}\,dt\\
&=\frac{1}{\lambda}+\inf_{\alpha_1(\cdot)}\Bigg\{ -\big[ \trajxo(t)e^{-\lambda t}\big]_0^{\tau(x_0)}
+(1-\lambda)\int_0^{\tau(x_0)} \trajxo(t)\,e^{-\lambda t}\,dt \Bigg\}\\
&=\frac{1}{\lambda}+ x_0 + \inf_{\alpha_1(\cdot)}\Big\{(1-\lambda)\int_0^{\tau(x_0)} \trajxo(t)\,e^{-\lambda t}\,dt \Big\}\,.
\end{aligned}$$
So, the best strategy here consists in choosing $\alpha_1\equiv-1$ to reach $\{x=0\}$ as fast as possible, then to stay there for later times.
This gives $\tau(x_0)=x_0$ and
$$\begin{aligned}
\VF_\mathrm{SC1}(x_0)&=\frac{1}{\lambda}+x_0+(1-\lambda)\int_0^{x_0} (x_0-t)\,e^{-\lambda t}\,dt\\
 &= \frac{x_0+1}{\lambda}-\frac{1-\lambda}{\lambda^2}\big(1-e^{-\lambda x_0}\big)\,.
\end{aligned}$$
Hence $\VF_\mathrm{SC1}(0)=1/\lambda^2>0$ for $\lambda\geq1$, and $\VF_\mathrm{SC1}(0)=1/\lambda>0$ for $\lambda<1$.
Here also, by symmetry we have $\VF_\mathrm{SC2}(x)=\VF_\mathrm{SC1}(-x)$ for $x\leq0$, so in any case, the ``pull-pull'' strategy is the best
$$\min\big\{u_\H(0),\VF_\mathrm{SC1}(0),\VF_\mathrm{SC2}(0)\big\}=u_\H(0)=0\,.$$
Then if we set
$$u_\sharp(x):=\begin{cases} \VF_\mathrm{SC1}(x) & \text{if }x>0\,,\\
\VF_\mathrm{SC2}(x) & \text{if }x\leq0\,,\end{cases}$$
by Theorem \ref{thm:structure} we have a subsolution (and actually a solution) in $\R$. But of course this solution does not satisfy the condition
$$u_\sharp(0)\leq\min\big\{u_\H(0),\VF_\mathrm{SC1}(0),\VF_\mathrm{SC2}(0)\big\}\,.$$
Notice however that for any $\lambda>0$ we always have $u_\sharp(0)\leq u_\H^\mathrm{reg}(0)$, so that
$$u_\sharp(0)=\min \big\{u_\H^\mathrm{reg}(0),\VF_\mathrm{SC1}(0),\VF_\mathrm{SC2}(0)\big\}=\VFp(0)\,.$$
Hence by uniqueness
in $\Omega_1$ and $\Omega_2$, we conclude that $u_\sharp\equiv\VFp$.

In order to prove non-uniqueness let us compute explicitly the minimal solution $\VFm$.
For $x>0$ we have, denoting by $\tau(x_0)$ the exit time for the trajectory starting from $x_0$
$$\begin{aligned}
    \VFm(x_0)&=\inf_{\alpha_1(\cdot)}\int_0^{\tau(x_0)}\bigg(1-\alpha_1(t)+\trajxo(t)\bigg)e^{-\lambda t}\,dt\\
     &=\frac{x_0}{\lambda}+\inf_{\alpha_1(\cdot)}\Bigg\{\frac{1-e^{-\lambda\tau(x_0)}}{\lambda}+
	\big(\frac{1}{\lambda}-1\big)\int_0^{\tau(x_0)}\alpha_1(t)\,e^{-\lambda t}\,dt\Bigg\}\,,\\
\end{aligned}$$
where we are using here the ``pull-pull'' strategy on $\H=\{x=0\}$ which keeps the trajectory at $x=0$ for a null cost after $\tau(x_0)$.

If $\lambda\leq 1$, the optimal control is obtained for $\alpha=-1$ which minimizes at the same time
$\tau(x_0)$ and the integral multiplied by $(1/\lambda-1)$. For instance, for $\lambda=1$ we get
$$\VFm(x_0)=x_0+\int_0^{x_0} e^{-t}\,dt=1+x_0-e^{-x_0}<u_\sharp(x_0)=x_0+1\,.$$
Of course the solution is symmetric for $x_0<0$.
In the case $\lambda<1$, the explicit solution can be computed the same way with $\alpha\equiv-1$
which gives
$$\VFm(x)=\frac{|x|}{\lambda}+\frac{2\lambda-1}{\lambda^2}\Big(1-e^{-\lambda |x|}\Big)\,.$$
The related Hamiltonian is computed as follows: for $x>0$ we have
$$\begin{aligned}
H_1(x,u,u_x)&=\sup_{\alpha_1}\bigg\{-\alpha_1 u_x + \lambda u -\big(1-\alpha_1+x\big)\bigg\}\\
&=\sup_{\alpha_1}\bigg\{-\alpha_1 (u_x-1)\bigg\} + \lambda u -x-1\\
&= |u_x-1|+\lambda u -x-1\,.
\end{aligned}$$
It can be checked that for all $\lambda\leq1$, $\VFm$ is indeed
the solution of $H_1=0$ with $\VFm(0)=0$. Finally, since $\VFm<\VFp$, uniqueness fails in this situation.

\section{Approximations, convergence} \label{sec:approx}

Since Problem~(\ref{Bellman-Om})-(\ref{Bellman-H-sub})-(\ref{Bellman-H-sup}) does not have a unique solution, the convergence of approximation schemes is not clear in general : indeed some approximations may convergence to $\VFm$, some others to $\VFp$ and we may even have convergence to other solutions. We will consider below two different cases. The first one is the case of Filippov's approximations; since it intuitively corresponds to a relaxation (where, roughly speaking, a larger set of controls is used), the answer is rather clear: we have convergence to the minimal solution $\VFm$. On the contrary, we have no general answer for the vanishing viscosity method since we have no simple interpretation which may indicate an approximation from above or below and therefore a convergence to either $\VFm$ or $\VFp$ (see however our conjecture below).

\subsection{Filippov's approximation}

A natural approximation of the above problem consists in introducing a continuous, increasing function $\varphi : \R \to \R$ such that
$$ \lim_{s \to -\infty}\varphi(s) = 0\quad \hbox{and}\quad\lim_{s \to +\infty}\varphi(s) = 1\; ,$$
and to study the behavior of the solution $\ue :\R^N \to \R$ of
\begin{equation}\label{eqn:vpe}
 \vpe (x_N) H_1(x,\ue, D\ue)+ (1-\vpe (x_N)) H_2(x,\ue, D\ue)=0\quad \hbox{in  }\R^N\; ,
\end{equation}
where $\vpe (x_N):=\varphi(x_N/\eps)$.

Contrary to the vanishing viscosity approach (see below), this method keeps record on what is happening on the hyperplane by "spreading" it
and tracks the controls that fulfill the compatibility condition between the two vector fields $b_1(x,\alpha_1)$ and $b_2(x,\alpha_2)$. Hence,
even the singular strategies are taken into account in the limit so that we obtain $\VFm$.

\begin{theo}  \label{cv-vpe}
Assume [H0], [H1] and [H2]. There exists a unique Lipschitz continuous solution $\ue$ of (\ref{eqn:vpe}). Moreover, as $\e \to 0$, $\ue \to \VFm$ locally uniformly in $\R^N$.
\end{theo}

\begin{proof} The first part of the theorem is clear since the Hamiltonian of Equation (\ref{eqn:vpe}) is coercive: by standard arguments, it is straightforward to obtain the existence and uniqueness of the $\ue$'s and to prove that they are equibounded and equi-Lipschitz continuous.

Applying Ascoli's Theorem, we may assume that the sequence $(\ue)_\e$ converges locally uniformly to a bounded, Lipschitz continuous function $u$ and it is also easy to show that $u$ satisfies (\ref{Bellman-Om})-(\ref{Bellman-H-sub})-(\ref{Bellman-H-sup}).

In order to conclude, we just have to show that $u$ is also a subsolution of (\ref{Bellman-H}). Indeed, if this is true, the result follows from Corollary~\ref{sous-max-sur-min} (iii).

Let $\phi=\phi(y')$ be a smooth function and let $x'$ be a strict local maximum point of $u(y',0)-\phi(y')$. We have to prove
$$\HT(x,u(x),D_\H\phi(x'))\leq 0\; ,$$
where $x=(x',0)$, i.e.
$$ - b_\H(x,\a)\cdot D_\H\phi(x')  + \lambda u (x) -  l_\H(x,\a)\leq 0\; ,
$$
for any $\a=(\aluno,\aldue,\mu) \in A_0(x)$ for which $(\mu b_1(x,\aluno)+(1-\mu)b_2(x,\aldue)) \cdot e_N=0$.

Since $\varphi$ is a continuous, increasing function, there exists $s\in \R$ such that $\varphi(s) = \mu$. We introduce the function
$$ \ue(y)-\phi(y') - \frac{1}{\e}|\frac{y_N}{\e}-s|^2\; .$$
By standard arguments, this function achieves a local maximum at a point $\xe$ close to $x$ and when $\e \to 0$, we have
$$\frac{1}{\e}|\frac{(\xe)_N}{\e}-s|^2 \to 0\; .$$
Moreover, using that $\ue$ is Lipschitz continuous, the derivative of this term $\displaystyle d_N := \frac{2}{\e^2}(\frac{(\xe)_N}{\e}-s)$ is bounded.

Now we write the viscosity subsolution inequality
$$
 \vpe ((\xe)_N) H_1(\xe,\ue(\xe), D_\H\phi(\xe')+d_Ne_N)+ (1-\vpe ((\xe)_N)) H_2(\xe,\ue (\xe), D_\H\phi(\xe')+d_Ne_N)\leq 0\; ,
$$
and we notice that, on one hand,
$$ \vpe ((\xe)_N) = \varphi(s) + o(1)\; ,$$
and, on the other hand, the $H_i$ terms are bounded since the gradients are bounded.

This yields
$$
 \varphi (s) H_1(\xe,\ue(\xe), D_\H\phi(\xe')+d_Ne_N)+ (1- \varphi (s) ) H_2(\xe,\ue (\xe), D_\H\phi(\xe')+d_Ne_N)\leq o(1)\; ,
$$
and, using the form of the $H_i$'s
$$ \varphi (s)\left(-b_1(\xe,\aluno)\cdot (D_\H\phi(\xe')+d_Ne_N)  + \lambda \ue (\xe) -  l_1(\xe,\aluno)\right)
+ $$
$$(1- \varphi (s) ) \left(-b_2(\xe,\aldue)\cdot (D_\H\phi(\xe')+d_Ne_N)  + \lambda \ue (\xe) -  l_2(\xe,\aldue)\right)\leq o(1)\; .
$$
Because of the choice of $s$ and since $\a=(\aluno,\aldue,\mu) \in A_0(x)$, this inequality is nothing but
$$ -(\mu b_1(\xe,\aluno)+(1-\mu)b_2(\xe,\aldue))\cdot D_\H\phi(\xe') + \lambda \ue (\xe) -  (\mu l_1(\xe,\aluno)+(1-\mu)l_2(\xe,\aldue))
\leq o(1)\; .
$$
And the conclusion follows by letting $\e$ tends to $0$.
\end{proof}

\subsection{Vanishing viscosity approximation}

In this section, we show that the vanishing viscosity approximation may converge to $\VFp$ by coming back to the example of Subsection \ref{subsec:pull} where non-uniqueness happens because of some singular (``pull-pull'') strategies which give a lower cost. Such strategies are rather instable and it is natural to think that, if we add a brownian perturbation, the trajectories will naturally tend to go away from $x={0}$. From the pde viewpoint, this instability is reflected in the fact that the vanishing viscosity method does not give $\VFm$ in the limit, but $\VFp$. More precisely we have

\begin{pro}\label{pro:viscous}
    Let us assume that we are in the framework of Subsection \ref{subsec:pull} and, for any $\eps>0$, consider
    the solution $u_\eps$ of the following problem
    \begin{equation}\label{pb:viscous}
        -\eps u_\eps'' + H(x,u_\eps,u_\eps') = 0\quad\text{in}\quad\R\,,
    \end{equation}
    where $H=H_1$ in $\Omega_1$ and $H_2$ in $\Omega_2$.
    Then, as $\eps\to0$, the sequence $(u_\eps)_\eps$ converges locally uniformly to $\VFp$ in $\R$.
\end{pro}

Before proceeding with the proof, let us precise that by a solution $u_\eps$, we mean a distributional solution
$u_\eps\in W^{1,\infty}_\mathrm{loc}(\R)$.
In particular, the possible discontinuity of $H$ when crossing $\H$ is not a problem in the integrated version of the equation:
for any $\varphi\in C^1_0(\R)$,
$$\eps\int_{\R}u_\eps'\varphi' + \int_{\R}H(x_\eps,u_\eps,u_\eps')\varphi=0\,.$$
This explains why we will recover in the limit only the strategies already encoded in the equations in $\Omega_1$ and $\Omega_2$,
and not the singular ones.

\begin{lem}
    For any $\eps>0$, there exist a unique solution of \eqref{pb:viscous} $u_\eps$ in $W^{2,r}_\mathrm{loc}(\R)$ for any $r>1$, which has a at most linear growth.
\end{lem}
\begin{proof} The proof follows classical methods and we are just going to sketch it. For more details, we refer the reader to the book of P.L Lions \cite{L} where similar results are obtained. The easiest way to prove the existence of  $u_\eps$ is by using first a  Filippov-type approximation: in this way, the nonlinearity becomes continuous w.r.t. all variables and, since $H_1, H_2$ are Lipschitz continuous (therefore at most linear in $p$), one easily builds a solution which is in $W^{2,r}_\mathrm{loc}(\R)$ for all $r>1$ and which grows at most linearly. We point out that a (uniform in $\e$) linear growth can be obtained by remarking that, for $K>0$ large enough, $\pm K(|x|^2+1)^{1/2}$ are respectively sub and supersolutions of \eqref{pb:viscous}.\end{proof}

\begin{proof}[Proof of Proposition \ref{pro:viscous}]
    We notice first that the solution $\VFp$ which is computed in Subsection \ref{subsec:pull} is always convex. Indeed, this is clear if $\lambda>1$
    since in this case $\VFp(x)=|x|/\lambda+1/\lambda^2$ and a straightfoward calculus shows that for $\lambda<1$,
    $(\VFp)''(x)= \frac{2\delta_0}{\lambda}+(1-\lambda)e^{-\lambda|x|}\geq0$ in the sense of distributions so that
    $$-\eps(\VFp)''(x)+H\big(x,\VFp(x),(\VFp)'(x)\big)=-\eps(\VFp)''(x)\leq0\,.$$
    Now we consider
    $w:=\VFp-u_\eps$. Substracting the inequations, and using the Lipschitz continuity of $H(x,u,p)$ in $p$ (for all $u$ and a.e. in $x$), there exists a constant $C>0$ such that
    $$-\eps w'' + \lambda w - C|w'|\leq0\quad\text{in}\quad\R\,.$$
    We first notice that both $\VFp$ and $u_\eps$ grow at most linearly, so does $w$. If we set
    $w_\eta:=(w-\eta(|x|^2+1))_+$ for some small $\eta>0$, then $w_\eta$ is compactly supported and therefore in $W^{1,\infty}(\R)$. Moreover, by similar arguments as in the proof of Lemma~\ref{soussol-inf}, $w_\eta$ still satisfies
    $$-\eps w_\eta'' + \lambda w_\eta - C|w_\eta'|\leq0\quad\text{in }\;\R\,,$$ at least if $\eps$ is small enough.

    We consider now a large integer $n$ to be chosen later. By using an approximation of $w_\eta^{2n+1}$ by test functions $\varphi_k\in C^1_0(\R)$ ($k\in \N$), we can pass to the limit as $k \to +\infty$
    in the weak formulation and get
    $$\eps(2n+1)\int_\R |w_\eta'|^2w_\eta^{2n}+\lambda\int_\R w_\eta^{2n+1}w_\eta-C\int_\R |w_\eta'|w_\eta^{2n+1}\leq0 \,.$$
    Using that
    $$|w_\eta'|w_\eta^{2n+1}=\Big(|w_\eta'|w_\eta^{n}\Big)\cdot\Big(w_\eta^{n+1}\Big)\leq
    \kappa \Big(|w_\eta'|^2w_\eta^{2n}\Big)+\frac{1}{\kappa}\Big(w_\eta^{2n+2}\Big)\,,$$
    we choose $\kappa=2C/\lambda$ and obtain
    $$\Big(\eps(2n+1)-\frac{2C}{\lambda}\Big)\int_\R |w_\eta'|^2w_\eta^{2n}
    +\frac{\lambda}{2}\int_\R w_\eta^{2n+1}w_\eta\leq0\,.
    $$
    Hence if choose $n$ large enough, we get $w_\eta\equiv 0$. Passing to the limit as $\eta\to0$ we find that $w\leq0$,
    which means $\VFp\leq u_\eps$ in $\R$.

    Finally we pass to the limit as $\eps\to0$ by using the half-relaxed limit method: if
$$
\overline u(x):={\limsup}^* u_\eps (x)=\limsup_{
\scriptsize
\begin{array}{l}
y  \to x\\
\eps \to 0
\end{array}}
u_\eps(y) \quad \hbox{and}\quad \underline u(x):={\liminf}^* u_\eps (x)=\liminf_{
\scriptsize
\begin{array}{l}
y  \to x\\
\eps \to 0
\end{array}}
u_\eps(y)\,,
$$
then $\overline u$ is a subsolution of (\ref{Bellman-Om})-(\ref{Bellman-H-sub})-(\ref{Bellman-H-sup}) and therefore, by Corollary~\ref{sous-max-sur-min}, $\overline u \leq \VFp$ in $\R$. But, on the other hand, $\VFp\leq u_\eps$ in $\R$ and this gives $\VFp\leq  \underline u \leq \overline u $ in $\R$. Therefore  $\VFp =  \underline u = \overline u $ in $\R$ which implies the uniform convergence of $u_\eps$ to $\VFp $.
\end{proof}

We notice that the same result holds under the assumptions of Subsection \ref{subsec:push}, but of course in this
case, $\VFm\equiv\VFp$.

\subsection{A conjecture }

Another approximation can be used through a combination between the ``vanishing viscosity method'' and the  Filippov's method
to obtain the approximate problem
\begin{equation}\label{eqn:vpeandVS}
-\delta_\e \Delta \ue + \vpe (x_N) H_1(x,\ue, D\ue)+ (1-\vpe (x_N)) H_2(x,\ue, D\ue)=0\quad \hbox{in  }\R^N\; ,
\end{equation}
where $\delta_\e$ is a parameter devoted to tend to zero. Of course, if $\delta_\e \ll \e$,
the expected behavior of $\ue$ is as in the  Filippov case, i.e. a convergence to $\VFm$ and we think
that if $\delta_\e \gg \e$, then $\ue$ converges to $\VFp$ as in the viscous approximation.

\

\

{\bf Acknowledgments.} We are grateful to P. Cardaliaguet for very fruitful  discussions on the definition of trajectories for the optimal control problem,  and to N. Tchou for interesting remarks on the first version of the manuscript, which led to valuable improvements of the paper.



\thebibliography{}

\bibitem{AF} J-P. Aubin and H. Frankowska,
Set-valued analysis. Systems \& Control: Foundations \& Applications, 2. Birkhuser Boston, Inc., Boston, MA, 1990.

\bibitem{ACCT} Y. Achdou, F. Camilli, A. Cutri, and N. Tchou, Hamilton-Jacobi  equations constrained   on networks, NDEA Nonlinear Differential Equation and Application,  to apper.


\bibitem{AMV}
Adimurthi, S. Mishra and G. D. Veerappa Gowda,
Explicit Hopf-Lax type formulas for Hamilton-Jacobi equations and conservation laws with discontinuous coefficients. 
J. Differential Equations 241 (2007), no. 1, 1-31.


\bibitem{BCD} M. Bardi and  I. Capuzzo Dolcetta, {\it Optimal control and viscosity solutions of Hamilton-Jacobi- Bellman equations}, Systems \& Control: Foundations \& Applications, Birkhauser Boston Inc., Boston, MA, 1997.

\bibitem{Ba} G. Barles,  {\it Solutions de viscosit\'e des  \'equations de Hamilton-Jacobi}, Springer-Verlag, Paris, 1994.

\bibitem{BJ:Rate}
G.~Barles and E.~R.~Jakobsen.
\newblock On the convergence rate of approximation schemes for
Hamilton-Jacobi-Bellman equations.
\newblock {\em M2AN Math. Model. Numer. Anal.} 36(1):33--54, 2002.


\bibitem{BP2}
G. Barles and B. Perthame: {\sl Exit time problems in optimal
control and vanishing viscosity method.} SIAM J. in Control and
Optimisation, 26, 1988, pp. 1133-1148.


\bibitem{BP3}
G. Barles and B. Perthame: Comparison principle for Dirichlet
type Hamilton-Jacobi Equations and singular perturbations of degenerated
elliptic equations. Appl. Math. and Opt., {\bf 21}, 1990, pp~21-44.


\bibitem{Bl1} A-P. Blanc,
Deterministic exit time control problems with discontinuous exit costs. SIAM J. Control Optim. 35 (1997), no. 2, 399--434.


\bibitem{Bl2} A-P. Blanc,
Comparison principle for the Cauchy problem for Hamilton-Jacobi equations with discontinuous data. Nonlinear Anal. 45 (2001), no. 8, Ser. A: Theory Methods, 1015--1037.

\bibitem{BrYu} A. Bressan and Y. Hong,  Optimal control problems on stratified domains, Netw. Heterog. Media 2 (2007), no. 2, 313-331 (electronic).

\bibitem{CaSo} F Camilli and A. Siconolfi, {Time-dependent measurable Hamilton-Jacobi equations}, Comm. in Par. Diff. Eq. 30 (2005), 813-847.

\bibitem{CR} G. Coclite and N. Risebro, Viscosity solutions of Hamilton-Jacobi equations with discontinuous coefficients.
J. Hyperbolic Differ. Equ. 4 (2007), no. 4, 771--795.

\bibitem{DeZS} C. De Zan and P. Soravia, Cauchy problems for noncoercive Hamilton-Jacobi-Isaacs equations with discontinuous coefficients.
Interfaces Free Bound. 12 (2010), no. 3, 347--368.

\bibitem{DE} K. Deckelnick and C. Elliott,
Uniqueness and error analysis for Hamilton-Jacobi equations with discontinuities.
Interfaces Free Bound. 6 (2004), no. 3, 329--349.

\bibitem{Du} P. Dupuis, A numerical method for a calculus of variations problem with discontinuous integrand. Applied stochastic analysis (New Brunswick, NJ, 1991), 90--107, Lecture Notes in Control and Inform. Sci., 177, Springer, Berlin, 1992.


\bibitem{Fi} A.F. Filippov,  Differential equations with discontinuous right-hand side. Matematicheskii Sbornik,  51  (1960), pp. 99--128.  American Mathematical Society Translations,  Vol. 42  (1964), pp. 199--231 English translation Series 2.

\bibitem{GS1}
M. Garavello and P. Soravia, Optimality principles and uniqueness for Bellman equations of unbounded control problems with discontinuous running cost. NoDEA Nonlinear Differential Equations Appl. 11 (2004), no. 3, 271-298.


\bibitem{GS2}
M. Garavello and P. Soravia,
Representation formulas for solutions of the HJI equations with discontinuous coefficients and existence of value in differential games.
J. Optim. Theory Appl. 130 (2006), no. 2, 209-229.

\bibitem{GGR}
Y. Giga, P. G\`orka and P. Rybka, A comparison principle for Hamilton-Jacobi equations with discontinuous Hamiltonians. Proc. Amer. Math. Soc. 139 (2011), no. 5, 1777-1785.


 \bibitem{gt}
 D. Gilbarg and N.S. Trudinger: {\sc Elliptic Partial Differential
 Equations of Second-Order.} Springer, New-York, (1983).

\bibitem{L}
Lions P.L.  {\it Generalized Solutions of Hamilton-Jacobi Equations},
Research Notes in Mathematics 69, Pitman, Boston, 1982.

\bibitem{R} R.T. Rockafellar,
{\it Convex analysis},
Princeton Mathematical Series, No. 28 Princeton University Press, Princeton, N.J. 1970.

\bibitem{IMZ} C. Imbert, R. Monneau and H. Zidani, A Hamilton-Jacobi approach to junction problems and applications to traffic flows, ESAIM COCV, to appear.

\bibitem{Son} H.M. Soner, Optimal control with state-space constraint,  I, SIAM J. Control Optim. 24 (1986), no. 3, 552-561.

\bibitem{ScCa} D. Schieborn and F.Camilli : Viscosity solutions of Eikonal equations on topological networks, to appear in Calc. Var. Partial Differential Equations.

\bibitem{So} P. Soravia,  Degenerate eikonal equations with discontinuous refraction index, ESAIM Control Optim. Calc. Var. 12 (2006).

\bibitem{Wa} T. Wasewski,  Syst\`emes de commande et \'equation au contingent, Bull. Acad. Pol. Sc., 9, 151-155, 1961.
\end{document}